\documentclass[12pt, oneside]{amsart}
\usepackage{appendix}
\usepackage{graphicx}
\usepackage{chngpage}
\usepackage{multirow}
\usepackage{hyperref}
\usepackage[all,cmtip]{xy}
\usepackage{amssymb}
\theoremstyle{plain}
\newtheorem{theorem}{Theorem}[section]
\newtheorem{lemma}[theorem]{Lemma}
\newtheorem{proposition}[theorem]{Proposition}
\newtheorem{corollary}[theorem]{Corollary}
\usepackage{fancyhdr}
\theoremstyle{definition}
\newtheorem{definition}[theorem]{Definition}

\theoremstyle{remark}
\newtheorem{remark}[theorem]{Remark}

\begin{document}

   \title{On smooth and isolated curves in general complete intersection Calabi-Yau threefolds}

\author{Xun Yu}
\address{Department of Mathematics, The Ohio State University, Columbus, OH, 43210-1174, USA}
\email{yu@math.ohio-state.edu}

\begin{abstract}
Recently Knutsen found criteria for the curves in a complete linear system $|\mathcal{L}|$ on a smooth surface $X$ in a nodal K-trivial threefold $Y_0$ to deform to a scheme of finitely many smooth isolated curves in a general deformation $Y_t$ of $Y_0$. In this article we develop new methods to check whether  the set of nodes of $Y_0$ imposes independent conditions on $|\mathcal{L}|$. As an application, we find new smooth isolated curves in complete intersection Calabi-Yau threefolds.

\end{abstract}

\maketitle
\section*{0. Introduction}

Calabi-Yau threefolds $Y$ have two related and interesting properties: a) $Y$ is unobstructed; b) the expected dimension of the deformation space of any l.c.i. curve lying in $Y$ is zero.

However, it is difficult to show even the existence of a rational curve of given degree and genus on $Y$, let alone that such a curve is geometrically rigid. Thus, there is interest in measuring the known families of geometrically rigid curves of given degree and genus on general $Y$. For genus zero it is known that there exist rigid rational curves of given degree on the general member of many complete intersection Calabi-Yau (CICY) threefolds $Y$. For higher genus Knutsen has provided many examples in [5]. Knutsen's technique is to construct a curve $C$ of given degree and genus lying in a linear system on a smooth surface $X$ lying in a nodal complete intersection Calabi-Yau threefold $Y_0$. He then uses deformation theory to show that only a finite number of the $C$ deform when $Y_0$ is deformed generally. 

In [5] he lists and proves a set of conditions sufficient to ensure this construction.

To introduce Knutsen's criterion, we first state the assumptions.

$\bold{Setting \; and \;assumptions}.$ Let $P$ be a smooth projective variety of dimension $r\geq 4$ and $\mathcal{E}$ a vector bundle of rank $r-3$ on $P$ that splits as a direct sum of line bundles $\mathcal{E}=\oplus_{i=1}^{r-3}\mathcal{M}_i$.

Let  $s_0=s_{0,1}\oplus ...\oplus s_{0,r-3}\in H^0(P, \mathcal{E})=\oplus_{i=1}^{r-3}H^0(P, \mathcal{M}_i)$ be a regular section, where $s_{0,i}\in H^0(P, \mathcal{M}_i)$ for $i=1,..., r-3$. Set $Y=Z(s_0)$ and $Z=Z(s_{0,1}\oplus ...\oplus s_{0, r-4})$\\ (where $Z=P$ if $r=4$).

Let $X\subset Y$ be a smooth , regular surface (i.e. $H^1(X, \mathcal{O}_X)=0$) and $\mathcal{L}$ a line bundle on $X$. 

We make the following additional assumptions:\\ (A1) $Y$ has trivial canonical bundle; \\(A2) $Z$ is smooth along $X$ and the only singularities of $Y$ which lie in $X$ are $l$ nodes $\xi_1, ..., \xi_l$. Furthermore $$ l\geq dim|\mathcal{L}+2|;$$\\ (A3) $|\mathcal{L}|\neq\emptyset$ and the general element of $|\mathcal{L}|$ is a smooth, irreducible curve;\\ (A4) for every $\xi_i\in S:=\{\xi_1,...,\xi_l\},$ if $|\mathcal{L}\otimes \mathcal{J}_{\xi_i}|\neq \emptyset$, then its general member is nonsingular at $\xi_i$; \\ (A5) $H^0(C, \mathcal{N}_{C/X})\cong  H^0(C, \mathcal{N}_{C/Y})$ for all $C\in |\mathcal{L}|$;\\ (A6) $H^1(C, \mathcal{N}_{C/P})=0$ for all $C\in |\mathcal{L}|$;\\ (A7) the image of the natural restriction map $$H^0(P, \mathcal{M}_{r-3})\longrightarrow H^0(S, \mathcal{M}_{r-3}\otimes\mathcal{O}_S)\cong \mathbb{C}^l$$

has codimension one.

Then Knutsen's criterion is the following:
\begin{theorem}

([5, Theorem1.1]). Under the above setting and assumptions (A1)-(A7), the members of $|\mathcal{L}|$ deform to a length 
\[\left( \begin{array}{c}
l-2  \\
 dim|\mathcal{L}| \end{array} \right) \] 
scheme of curves that are smooth and isolated in the general deformation $Y_t=Z(s_0+ts)$ of $Y_0=Y$. In particular, $Y_t$ contains a smooth, isolated curve that is a deformation of a curve in $|\mathcal{L}|$.

\end{theorem}

\vspace{2 mm}

Using Knutsen's criterion, we find some new smooth and isolated curves in general Calabi-Yau complete intersection (CICY) threefolds in this paper. 


To apply Theorem 0.1, we need to choose appropriate surfaces $X$ and line bundles $\mathcal{L}$ on $X$. Then we need to show that we can find a nodal Calabi-Yau threefold $Y$ containing $X$ such that all the  conditions (A1)-(A7) are satisfied. [5, Proposition 4.3] shows that under certain assumptions,  (A5) is equivalent to (A5'). Actually, all cases in this paper satisfy those assumptions and we will always check (A5') instead of (A5). (A5') consists of two parts:

\vspace{1mm}

$(1)$ The set of nodes $S$ imposes independent conditions on $|\mathcal{L}|$; ($\star$)

$(2)$  The natural map $\gamma_{C} :H^0(C, \mathcal{N}_{X/Y}\otimes \mathcal{O}_C) \longrightarrow H^1(C, \mathcal{N}_{C/X})$(cf.[5] (4.4)) is an isomorphism for all $C\in |\mathcal{L}|$.

\vspace{1mm}

The part (1) of (A5') is critical to this paper, and we denote it as $(\star)$. 

In \S 1 we develop the tools we will need to prove $(\star)$ in our cases. In \S 2 we treat cases in which $X$ is a K3 surface. Finally in \S 3 we treat the cases in which $X$ is a rational surface.

\vspace{2mm}

Our main results are Theorem 2.9, Theorem 2.12 and Theorem 3.1. 

\vspace{2mm}
$\mathbf{Theorem \;2.9}$: Let $d\geq 1$ and $g\geq 0$ be integers. Then in any of the following cases the general Calabi-Yau complete intersection threefold $Y$ of a particular type contains an isolated, smooth curve of degree $d$ and genus $g$:

\vspace{1mm}

(a) $Y=(5)\subset \mathbb{P}^4:  g=23$ and $d=18$; $g=24$ and $d=19$; $g=26$ and $d=20$; $g=27$ and $d=20$; $g=29$ and $d=21$.

(b)  $Y=(4,2)\subset \mathbb{P}^5:$ $g=16$ and $17\le d\le 19$; $g=17$ and $17\le d\le 20$; $g=18$ and $d=20$; $g=19$ and $d=18,20 \; and\; 21$; $g=20$ and $d=19, 20$; $g=21$ and $d=20, 21$; $g=22$ and $d=20, 21$; $g=23$ and $d=20, 22$; $g=25$ and $d=21, 22$; $g=26$ and $d=22$; $g=27$ and $d=22$; $g=29$ and $d=23$.

(c) $Y=(3,3)\subset \mathbb{P}^5: g=8$ and $d=12$.

(d) $Y=(2,2,3 )\subset \mathbb{P}^6: g=11$ and $d=16$.

(e) $Y=(2,2,2,2 )\subset \mathbb{P}^7:$ $g=4$ and $d=9$; $g=5$ and $d=10, 11$.

\vspace{2mm}

$\mathbf{Theorem \;2.12}$: Let $d\geq 1$ and $g\geq 0$ be integers. Then in any of the following cases the general Calabi-Yau complete intersection threefold $Y$ of a particular type contains an isolated, smooth curve of degree $d$ and genus $g$:

\vspace{2mm}

(a) $Y=(5)\subset \mathbb{P}^4:  g=23$ and $d=19$; $g=24$ and $d=20$; $g=25$ and $d=19, 20$.

(b)  $Y=(4,2)\subset \mathbb{P}^5:$ $g=16$ and $d=20$; $g=18$ and $d=18, 19$; $g=19$ and $d=19$; $g=20$ and $d=21$; $g=21$ and $d=19$; $g=23$ and $d=21$;  $g=24$ and $d=21, 22$; $g=27$ and $d=23$; $g=28$ and $d=23$.

(c) $Y=(2,2,3 )\subset \mathbb{P}^6: g=4$ and $d=8$.

(d) $Y=(2,2,2,2 )\subset \mathbb{P}^7:$ $g=4$ and $d=10$; $g=6$ and $d=11$.

\vspace{1mm}

In order to use Theorem 0.1 to prove Theorem 2.9 and Theorem 2.12, the surfaces $X$ are complete intersection $K3$ surfaces with $PicX=\mathbb{Z}H\oplus\mathbb{Z}C$. (cf.[6, Theorem 1.1]), where $H$ is the hyperplane section of $X$ and $C$ is a smooth irreducible curve of desired genus and degree. The K3 surfaces used in the proof of Theorem 2.9 do not have -2 divisors, but the K3 surfaces used in the proof of Theorem 2.12 have -2 divisors. Furthermore, we define the line bundle $\mathcal{L}$ to be $\mathcal{O}_X(C)$.  Let $S$ be nodes of general CICY $Y$ containing $X$. By the construction of $X$ and $Y$,  $S=A\cap B$, where $(A,B)$ is a general member of  $|\mathcal{O}_X(a)|\times |\mathcal{O}_X(b)|$ and $a\leq b$. The integers $a$ and $b$ vary case by case. The good news is that for all cases in Theorem 2.9 and Theorem 2.12, all the conditions [5] (A1)-(A7) except the part (1) of (A5'), say ($\star$), can be easily verified by using results in [5, \S6 \& \S7] (Lemma 6.1, Lemma 6.2, Prop.6.5 and the proof of Prop 7.2). The bad news is that ($\star$) is hard to verify. Actually, in [5, Proposition 7.2] Knutsen uses  the  condition (7.2) to guarantee ($\star$), but all cases in Theorem 2.9 and Theorem 2.12 do not satisfy the condition (7.2). For this reason, in \S 1 and Appendix A we develop two new methods which can be used to verify ($\star$). 

Actually, we can appy Corollary 1.8 to show $(\star)$ for all cases in Theorem 2.9 and Theorem 2.12. In order to apply Corollary 1.8, we will need to show that $H^1(X, \mathcal{L}(-a-b))=0$. If the K3 surfaces $X$ do not have -2 divisors like in Theorem 2.9, we can just apply Lemma 2.8 to show  $H^1(X, \mathcal{L}(-a-b))=0$. Roughly speaking, Lemma 2.8 gives a numerical criterion for the vanishing of the first cohomology group of line bundles on a K3 surface without -2 divisors. Essentially, Lemma 2.8 comes from the fact that on a K3 surface without -2 divisors all complete linear system are base points free. On the other hand, if the K3 surfaces $X$ have -2 divisors, we first compute the closed cone of curves $\overline{NE(X)}$ and the nef cone $Nef(X)$ and then we can prove $H^1(X, \mathcal{L}(-a-b))=0$ by using the information about these cones.

\vspace{2mm}

$\mathbf{Theorem \; 3.1}$: Let $d\geq 1$ and $g\geq 0$ be integers. Then in any of the following cases the general Calabi-Yau complete intersection threefold $Y$ of a particular type contains an isolated, smooth curve of degree $d$ and genus $g$:

(i) $Y=(3,3)\subset \mathbb{P}^5: g=3$ and $d=6$.

(ii) $Y=(2,4)\subset \mathbb{P}^5: g=3$ and $d=6$.

\vspace{2mm}
The proof of Theorem 3.1 is a quite different story since in this case we use complete intersection rational surfaces instead of complete intersection K3 surfaces. For case (i) of Theorem 3.1 we choose $X=\mathbb{P}^2$ blow-up at six general points and $\mathcal{L}=\mathcal{O}_X(H+l)$, where the hyperplane class $H=$proper transform of a cubic through the six points and $l$ is the pull-back of $\mathcal{O}_{\mathbb{P}^2}(1)$; for case (ii) we choose $X=(2)\subset\mathbb{P}^3, \mathcal{L}=(2,4)\in \mathbb{Z}\times\mathbb{Z}\cong Pic\;X$.

\vspace{2mm}
To prove Theorem 3.1, we will also check all conditions (A1)-(A7) for each case. Notice that all these rational surfaces used are Fano varieties,  which makes verifying the conditions (A1)-(A7) much easier since some cohomology groups in question vanish by applying Kodaira Vanishing to -$K$.

(A1) and (A2) are easily verified. Because the line bundles $\mathcal{L}$ are very ample, (A3) and (A4) are easily verified. The way to check (A7) is similar to [5, Lemma 6.2]. The way to check (A6) is similar to certain parts of the proof of [5, Proposition 7.2]. Like before we check (A5') instead of (A5). The second part of (A5') is easily checked since it's easy to show $H^0(C,\mathcal{N}_{X/Y}\otimes\mathcal{O}_C)\cong H^1(C, \mathcal{N}_{C/X})=0,\forall C\in|\mathcal{L}|$. In order to show $(\star)$, we just apply Corollary 1.8.

Because this paper cites many results from reference [5] and some of them are not described here, the reader probably need a copy of  reference [5] in order to read the rest of this paper.

In the Appendix C, a summary for the existence of smooth isolated curves in general CICY threefolds known so far is given by putting results in [5] and this paper together.

\subsection*{Acknowledgments} I would like to thank my advisor Herb Clemens for his insight, advice and continuous support. This paper grew out of numerous conversations with him. I would also like to thank Andreas Knutsen for reading the first draft of this paper and for giving very helpful comments. One conversation with David Morrison was also very helpful, and is gratefully acknowledged.

\section{DEFINITIONS AND LEMMAS}

It is convenient to consider the first half of [5] (A5') i.e. ``the nodes $S$ imposes independent conditions on $|\mathcal{L}|$'', say ($\star$),  in a general setting.

\vspace{2 mm}
Let $X$ be a smooth surface in projective space and let $\mathcal{L}$ be a line bundle on $X$ such that $h^0(X,\mathcal{L})>1$. Let $n=h^0(X, \mathcal{L}).$

\begin{definition}
Let $S$ be a reduced 0-cycle on $X$. We say $S$ imposes independent conditions on $|\mathcal{L}|$ if $\forall C\in|\mathcal{L}|$, the natural evaluation map $H^0(X, \mathcal{L})\longrightarrow H^0(S\cap C, \mathcal{O}_{S\cap C})$ is surjective. 
\end{definition}

Clearly,  $S$ imposes independent conditions on $|\mathcal{L}|$ if and only if  $\forall$ subset $D\subset S$, the natural restriction map $H^0(X, \mathcal{L})\longrightarrow H^0(D, \mathcal{L}|_{D})$ of maximal rank.

\begin{remark}
Note that if $S$ imposes independent conditions on $|\mathcal{L}|$, then, in particular, the points in $S$ are different from the possible base points of $|\mathcal{L}|$, so that the locus of curves in $|\mathcal{L}|$ passing through at least one point of $S$ is an effective divisor in $|\mathcal{L}|$ consisting of hyperplanes. Therefore the condition that  $S$ imposes independent conditions on $|\mathcal{L}|$ can be rephrased as saying that the locus of curves in $|\mathcal{L}|$ passing through at least one point of $S$ is an effective, simple normal crossing divisor consisting of hyperplanes.

If the cardinality $|S| >n$, then  the condition that $S$ imposes independent conditions on $|\mathcal{L}|$ is equivalent to the condition that at most $dim|\mathcal{L}|=n-1$ points of $S$ can lie on an element of $|\mathcal{L}|$.
\end{remark}

For positive integers $a\le b$ let $$l=c_1(\mathcal{O}_X(a))\cdot c_1(\mathcal{O}_X(b)).$$

We  will assume through out that
\begin{equation}
l>n \; \text{and} \; H^0(X,\mathcal{L}(-a))=0.
\end{equation}

\begin{definition}
Let $A\in |\mathcal{O}_X(a)|$ be a smooth irreducible curve. We say the pair $(A, \mathcal{O}_X(b))$ can impose independent conditions on $|\mathcal{L}|$ if there exists $B\in |\mathcal{O}_X(b)|$ such that $B\cap A$ is a set of $l$ distinct points and $B\cap A$ imposes independent conditions on $|\mathcal{L}|.$
\end{definition}

\begin{remark}
 In later sections, the surface $X$ will be a smooth and complete intersection surface (K3 or rational). The nodes $S$ of a general CICY threefold $Y$ containing $X$  will be a complete intersection on $X$, say $S=A\cap B$ where $A\in |\mathcal{O}_X(a)|$ and $B\in  |\mathcal{O}_X(b)|$. The positive integers $a$ and $b$ are determined by complete intersection types of $X$ and $Y$. (cf. Table 1 in section 2).  We want to show $S$ imposes independent conditions on the chosen complete linear system $|\mathcal{L}|$ on $X$. Actually, our goal is to show that there is a dense subset $\mathcal{U}$ of $|\mathcal{O}_X(a)|\times|\mathcal{O}_X(b)|$ such that for any $(H_0,H_1)\in \mathcal{U}$ the intersection $H_0\cap H_1$ is a reduced 0-cycle on $X$ and it imposes independent conditions on $|\mathcal{L}|$.  Clearly  this is  reduced to show the following statement:  for any fixed smooth irreducible member $A_0\in |\mathcal{O}_X(a)|$ there exists a member $B\in |\mathcal{O}_X(b)|$ such that $A_0\cap B$ is a set of $l$ distinct points that impose independent conditions on $|\mathcal{L}|$. This simple observation motivates the above definition, and this reduction will allow us to bring to bear the classical theory of divisors on a Riemann surface.
 
 \end{remark}

 From now on, let $A_0\in |\mathcal{O}_X(a)|$ be any fixed irreducible smooth curve.

In this paper two different new methods that can be used to show ``the pair $(A_0, \mathcal{O}_X(b))$ can imposes independent conditions on $|\mathcal{L}|$. '' will be presented. The first one, that we call Method I,  is a sort of generalization of Knutsen's method in [5] (cf. [5, Lemma 6.3]). The second one, Method II, is completely new. For application purposes, in all cases in Theorem 2.9, Theorem 2.12 and Theorem 3.1 of this paper we apply Method I to show $(\star)$ without using Method II at all. Some cases in Theorem 2.9, Theorem 2.12 and Theorem 3.1 also follow from Method II but the others do not. However, theoretically it is  possible that in certain situations we can not apply Method I (for example, the condition i) of Corollary 1.8 is not satisfied.) but we can still apply Method II. It is hoped that Method II can find applications somewhere else. Because we only use Method I in \S 2 and \S 3, Method II is put in the Appendix A.

\subsection{Method I}

\begin{proposition}
 Define $W:=\{(D, B)\in A_{0}^{(n)}\times | \mathcal{O}_{A_0}(b)|| D\subset B \}$, where $A_{0}^{(n)}$ is the n-th symmetric product of $A_0$. We assume $H^1(X, \mathcal{O}_{X}(b-a))=0$. Then $(A_0,\mathcal{O}_X(b))$ can impose independent conditions on $|\mathcal{L}|$ if the following two conditions are satisfied:

 i) $W$ is irreducible;

ii) $\exists B_0\in |\mathcal{O}_{A_0}(b)|, \ni$ the following restriction map is injective $$ H^0(X,\mathcal{L})\longrightarrow H^0(B_0, \mathcal{L}\otimes\mathcal{O}_{B_0})$$

\end{proposition}

\begin{proof}
Suppose conditions i) and ii) are satisfied.

Define $W':=\{ (D,B)\in W | $ the restriction map  $ H^0(X,\mathcal{L})\longrightarrow H^0(D, \mathcal{L}\otimes\mathcal{O}_{D})$ is not injective$\}$. Let $p: W\longrightarrow |\mathcal{O}_{A_0}(b)|$ be the natural projection map. $dimW=dim|\mathcal{O}_{A_0}(b)|$ since $p$ is a finite surjective map.

Define $U=\{B\in |\mathcal{O}_{A_0}(b)| |$ the restriction map $ H^0(X,\mathcal{L})\longrightarrow H^0(B, \mathcal{L}\otimes\mathcal{O}_{B})$ is injective$  \}.$ By condition ii), $U$ is open dense in $|\mathcal{O}_{A_0}(b)|.$ So without loss of generality, we can assume the $B_0$ in ii) consists of $l$ distinct points.

Now because the restriction map $ H^0(X,\mathcal{L})\longrightarrow H^0(B_0, \mathcal{L}\otimes\mathcal{O}_{B_0})$ is injective, by linear algebra  $\exists $ subset $D_0\subset B_0$, $\ni D_0$ is a set of $n$ distinct points and the restriction map $H^0(X,\mathcal{L})\longrightarrow H^0(D_0, \mathcal{L}\otimes\mathcal{O}_{D_0})$ is injective too. Therefore, no member of $|\mathcal{L}|$ can contain $D_0$, which implies $(D_0, B_0)\in W$ but $(D_0, B_0)\notin W'.$ By i) $dim W'< dim W=dim|\mathcal{O}_{A_0}(b)|.$ So $W'$ cannot dominate $|\mathcal{O}_{A_0}(b)|$ via the map $p$. 

On the other hand, all divisors in the complementary of $p(W')$ in $|\mathcal{O}_{A_0}(b)|$ impose independent conditions on $|\mathcal{L}|$. By the assumption $H^1(X, \mathcal{O}_X(b-a))=0$, the natural restriction map $ H^0(X,\mathcal{O}_X(b))\longrightarrow H^0(A_0, \mathcal{O}_{A_0}(b))$ is surjective. So $(A_0,\mathcal{O}_X(b))$ can impose independent conditions on $|\mathcal{L}|.$ 
 
\end{proof}

The following general result gives us a nice criterion for the irreducibility of $W$.

\begin{lemma} Let $A$ be any smooth projective curve and $\mathcal{B}$ is a very ample line bundle on $A$. $deg\mathcal{B}=l$, $dim|\mathcal{B}|=N$. Define the incidence variety $W=\{(D,B)| D\subset B \}\subset A^{(n)}\times |\mathcal{B}|$, where $0\leq n\leq l$. Then $W$ is irreducible if $min\{n, l-n\}\leq N$. 

\end{lemma}

\begin{proof}
If $n=0$ or $l$, obviously $W$ is irreducible. 

Suppose $0<n\le N$. The complete linear system $|\mathcal{B}|$ induces an embedding $\phi : A\hookrightarrow \mathbb{P}^N$ with $\phi^{*}(\mathcal{O}_{\mathbb{P}^N}(1))=\mathcal{B}$  and  $\phi^{*}: H^0(\mathbb{P}^{N}, \mathcal{O}_{\mathbb{P}^N}(1))\longrightarrow H^0(A,\mathcal{B})$ is an isomorphism. Then $A\subset \mathbb{P}^N$ is a nondegenerate algebraic curve of degree $l$. Let $(\mathbb{P}^N)^*$ be the set of hyperplanes of $\mathbb{P}^N$. By abusing notation, we can identify $(\mathbb{P}^N)^*$ with $|\mathcal{B}|$, and use them interchangeably later on.

Define $U=\{H\in (\mathbb{P}^N)^* | 
H\cap A=\{x_1,x_2, ..., x_l \} \; l\;$ distinct points,  and any set of $ \; n\; \text{points}\; \{x_{i_1},..., x_{i_n}\} $ imposes $n$ independent conditions on $|\mathcal{B}|  \}$. By Castelnuovo's general position theorem, $U$ is open dense in $(\mathbb{P}^N)^*$.

 Let $W_{U}=\{(D, B)\in W | B\in U \}$. Fix any $(D_0, B_0)\in W$. Because $U$ is dense in $(\mathbb{P}^N)^{*}$, we can find a one-parameter family of hyperplanes $H_t\in (\mathbb{P}^N)^{*}$ such that: $H_0\cap A=B_0$, and $\forall t\in \bigtriangleup \setminus\{0\}, H_t\in U$. Clearly then we can find $D_t\subset (H_t\cap A)$, $\ni (D_t, H_t\cap A)$ specialize to $(D_0, B_0)$. Therefore, $(D_0, B_0)\in \overline{W_U}$, which implies $W_U$ is open dense in $W$.

\vspace{2mm}

In order to show $W$ is irreducible, we only need to show $W_U$ is irreducible. To this end, consider the projection map $q: W\longrightarrow A^{(n)}$, where $q(D,B)=D$. Let $V=\{D\in A^{(n)}| dim|\mathcal{B}(-D)|=N-n$, i.e. $D$ imposes $n$ independent conditions on $|\mathcal{B}|$ $ \}$. Clearly $V$ is open dense in $A^{(n)}$. Furthermore, $q^{-1}(V)$ is irreducible since $\forall D\in V,$ the fibers $q^{-1}(D)$ are irreducible and of the same dimension. Because $W_U\subset q^{-1}(V)$ open subset, $W_U$ is irreducible too.

\vspace{2mm}

Next, suppose $0<l-n\le N$.  Define $W'=\{(D',B)| D'\subset B \}\subset A^{(l-n)}\times |\mathcal{B}|$. By the argument above, we know that $W'$ is irreducible. However, obviously $W\cong W'$ so $W$ is irreducible too.

The lemma is proved. 

\end{proof}

\vspace{2mm}

\begin{remark}
 All cases in Knutsen's paper [5] satisfy $n\le N$ in Lemma 1.6. Actually,  in [5] the method used to prove $(\star)$ requires $n\le N$ (cf.[5, Lemma 6.3]).   All cases in this paper can only satisfy $l-n\le N$.

\end{remark}

\begin{corollary}
We assume $H^1(X, \mathcal{O}_{X}(b-a))=0$. Then $(A_0,\mathcal{O}_X(b))$ can impose independent conditions on $|\mathcal{L}|$ if the following two conditions are satisfied:

i) $\frac{l}{2}\leq dim|\mathcal{O}_{A_0}(b)|$;

ii) $H^0(A_0, \mathcal{L}\otimes \mathcal{O}_{A_0}(-b))=0$.

Furthermore,  $H^1(X, \mathcal{L}(-a-b))=0$ implies the condition ii).
\end{corollary}

\begin{proof}
$\frac{l}{2}\leq dim|\mathcal{O}_{A_0}(b)|$ implies either  $n\leq dim|\mathcal{O}_{A_0}(b)|$ or  $l-n\leq dim|\mathcal{O}_{A_0}(b)|$, so $W$ is irreducible by Lemma 1.6.

Clearly $H^0(A_0, \mathcal{L}\otimes \mathcal{O}_{A_0}(-b))=0$ implies the condition ii) in Proposition 1.5.

Next, we need to show $H^1(X, \mathcal{L}(-a-b))=0$ implies $H^0(A_0, \mathcal{L}\otimes \mathcal{O}_{A_0}(-b))=0$.To this end let's consider the following exact sequence of sheaves, $$0\longrightarrow \mathcal{L}(-a-b)\longrightarrow \mathcal{L}(-b)\longrightarrow \mathcal{L}\otimes\mathcal{O}_{A_0}(-b)\longrightarrow 0.$$
Taking cohomology groups, we have $H^0(X, \mathcal{L}(-b))\longrightarrow H^0(A_0, \mathcal{L}\otimes \mathcal{O}_{A_0}(-b))\longrightarrow H^1(X, \mathcal{L}(-a-b))) $ exact. By assumption  $H^0(X,\mathcal{L}(-a))=0$ and hence $ H^0(X,\mathcal{L}(-b))=0$ since $a\leq b$, so $H^1(X, \mathcal{L}(-a-b))=0$ implies $H^0(A_0, \mathcal{L}\otimes \mathcal{O}_{A_0}(-b))=0$.

\end{proof}

\begin{remark}
In \S 2 and \S 3, we will just use Corollary 1.8. to show $(\star)$ for all cases.
\end{remark}

\vspace{3mm}

\begin{center}
\section{Curves on K3 surfaces in nodal Calabi-Yau threefolds}
\end{center}

In the rest of this paper, we will use Theorem 0.1 to show the existence of  smooth and isolated curves in general complete intersection Calabi-Yau (CICY) threefolds. In this section, these curves are obtained by deforming a careful chosen continuous family of curves on a complete intersection K3 surface in a nodal CICY threefold.  

We first recall how we can embed a complete intersection K3 surface into a nodal CICY threefold. We will follow notations used in [5, \S 6]. 

It is well known that there are three types of complete intersection K3 surfaces in projective space, namely the intersection types (4) in $\mathbb{P}^3$, (2,3) in $\mathbb{P}^4$ and $(2,2,2)$ in $\mathbb{P}^5$. Similarly, there are five types of CICY threefolds in projective space, namely the intersection types $(5)$ in $\mathbb{P}^4$, $(3,3)$ and $(4,2)$ in $\mathbb{P}^5$,  $(3,2,2)$ in $\mathbb{P}^6$ and $(2,2,2,2)$ in $\mathbb{P}^7$.

\newpage
\centerline{Table 1} 
\begin{tabular}{|c|c|c|}
\hline 
$Y=(a_1,a_2,...,a_{r-4},a_{r-3}+a_{r-2})\subset \mathbb{P}^r$  &$X=(a_1, a_2,...,a_{r-2})\subset \mathbb{P}^r$  & $r$\\
\hline 
\hline 
 $(5)$ & $(4,1)$  & 4 \\
\hline 
  $(5)$ &$(3,2)$   & 4 \\
\hline
  $(4,2)$ &$(4,1,1)$   & 5 \\
\hline
 $(2,4)$ &$(2,3,1)$   &5\\
\hline 
$(2,4)$  & $(2,2,2)$  &5  \\
\hline 
$(3,3)$  &  $(3,2,1)$  & 5 \\
\hline 
$(3,2,2)$  &  $(3,2,1,1)$  &6  \\
\hline 
$(2,2,3)$ &  $(2,2,2,1)$  &6  \\
\hline 
$(2,2,2,2)$  & $(2,2,2,1,1)$   &7  \\
\hline 
\end{tabular}

\vspace{10 mm}

\begin{remark}
Table 1 here is a part of  [5, Table 1 in \S 6]. Notice that in [5, Table 1 in \S 6] the complete intersection types of $Y$ are denoted as $(b_i)$.
\end{remark}

Our goal is to embed a given smooth complete intersection K3 surface $X$ of type $(a_1, a_2,...,a_{r-2})$ into a nodal CICY threefold $Y$ of type $(a_1, a_2,...,a_{r-4}, a_{r-3}+a_{r-2})$. To this end, we first choose generators $g_i$ of degrees $a_i$ for the ideal of $X$. So $X=Z(g_1,...,g_{r-2}).$ 

Define $$f_i:=\sum\alpha_{ij}g_{j}$$ where $\alpha_{ij}$ are general in $H^0( \mathbb{P}^{r},\mathcal{O}_{\mathbb{P}^r}(a_i-a_j))$ if $1\leq i\leq r-4$ and  $\alpha_{(r-3)j}$ are general in $H^0( \mathbb{P}^{r},\mathcal{O}_{\mathbb{P}^r}(a_{r-3}+a_{r-2}-a_j))$.

 Then define $$Y:=Z(f_1,...,f_{r-3}).$$ 
 
 If the coefficient forms $\alpha_{ij}$ are chosen in a sufficient general way, $Y$ has only $l=a_1a_2...a_{r-4}a_{r-3}^2a_{r-2}^2$ ordinary double points and they lie on $X$. This can be checked using Bertini's theorem. In fact, the $l$ nodes are the intersection points of two general elements of $|\mathcal{O}_X(a_{r-2})|$ and $|\mathcal{O}_X(a_{r-3})|$ (distinct, when $a_{r-2}=a_{r-3}$). As above, we denote the set of nodes by $S$.

Moreover, for general $\alpha_{ij}$, Bertini's theorem yields that the fourfold $$Z:=Z(f_1,...,f_{r-4})$$ is smooth. (Note that $Z=\mathbb{P}^r$ if $r=4$.)

We are therefore in the setting of Theorem 0.1 with $P=\mathbb{P}^r$, $$\mathcal{E}:=(\oplus_{i=1}^{r-4}\mathcal{O}_{\mathbb{P}^r}(a_i))\oplus \mathcal{O}_{\mathbb{P}^r}(a_{r-3}+a_{r-2})$$ and $\mathcal{M}_{r-3}:=\mathcal{O}_{\mathbb{P}^r}(a_{r-2}+a_{r-3})$.

\vspace{3mm}
\begin{remark}
Actually, as mentioned in \S 1, the integers $a_{r-2}$ and $a_{r-3}$ correspond to the integers $a$ and $b$ used in \S 1.
\end{remark}

As mentioned in the introduction, in this section we will prove the main theorems Theorem 2.9 and Theorem 2.12.

To apply Theorem 0.1 to prove Theorem 2.9 and Theorem 2.12, we first need a smooth regular surface $X$ and a line bundle $\mathcal{L}$ on $X$. All surfaces $X$ which will be used in the proofs of Theorem 2.9 and Theorem 2.12 are complete intersection K3 surfaces as in Table 1 with $Pic\; X=\mathbb{Z}H\oplus \mathbb{Z}C$, where $H$ is the hyperplane section of $X$ and $C$ is a smooth irreducible curve on $X$. Furthermore, the genus $g$ and degree $d$ of $C$ are exactly the genus and degree of the desired smooth and isolated curves in general CICY threefolds. Then we define $\mathcal{L}$ to be $\mathcal{O}_X(C)$. The existence of these K3 surfaces are guaranteed by the following theorem due to Knutsen.
\vspace{2mm}

\begin{theorem}
 ([6, Theorem 1.1]). Let $n\geq 2, d>0, g\geq 0 $ be integers. Then there exists a K3 surface $X$ of degree $2n$ in $\mathbb{P}^{n+1}$ containing a smooth curve $C$ of degree d and genus $g$ if and only if
 
 \vspace{2mm}

 (i) $g=\frac{d^2}{4n}+1$ and there exist integers $k, m\geq 1$ and $(k,m)\neq (2,1)$ such that $n=k^2m$ and $2n$ divides $kd$,
 
 (ii) $\frac{d^2}{4n}<g<\frac{d^2}{4n}+1$ except in the following cases 
 
 \hspace{.2in} (a) $d\equiv \pm 1, \pm 2\; (mod\; 2n)$,
 
 \hspace{.2in} (b)  $d^2-4n(g-1)=1$ and $d\equiv n\pm 1\; (mod\; 2n)$,
 
 \hspace{.2in} (c) $d^2-4n(g-1)=n$ and $d\equiv n\; (mod\; 2n)$,
 
 \hspace{.2in} (d) $d^2-4n(g-1)=1$  and $d-1$ or $d+1$ divides $2n$,
 
 (iii) $g=\frac{d^2}{4n}$ and $d$ is not divisible by $2n$,
 
 (iv) $g<\frac{d^2}{4n}$ and $(d,g)\neq (2n+1, n+1)$.
 
 \vspace{3mm}
 
 Furthermore, in case (i) $X$ can be chosen such that $Pic \;X= \mathbb{Z}\frac{2n}{dk}C=\mathbb{Z}\frac{1}{k}H$ and in cases (ii)-(iv) such that $Pic\; X=\mathbb{Z}H\oplus \mathbb{Z}C$, where $H$ is the hyperplane section of $X$.
 
 If $n\geq 4, \; X$ can be chosen to be scheme-theoretically an intersection of quadrics in cases (i), (iii) and (iv), and also in case (ii), except when $d^2-4n(g-1)=1$ and $3d\equiv \pm 3\; (mod\; 2n)$ or $d^2-4n(g-1)=9$ and $d\equiv \pm 3\; (mod\; 2n)$, in which case $X$ has to be an intersection of both quadrics and cubics.

\end{theorem}

 In order to apply Theorem 0.1, we need to show that $X$ can be embedded into a nodal CICY threefold $Y$ such that the conditions (A1)-(A7) mentioned in the introduction are all satisfied. As explained in the proofs of Theorem 2.9 and Theorem 2.12, all of them except the condition $(\star)$ can be easily verified by using results in [5]. We will use Corollary 1.8 to check $(\star)$.

\begin{definition}
Let X be a K3 surface. A divisor $D\in Div(X)$ is called -2 divisor if the self-intersection $D^2=-2$.
\end{definition}

\begin{remark}
It is easy to see that a K3 surface $X$ has a -2 divisor if and only if it contains a smooth rational curve.
\end{remark}
In order to use Corollary 1.8, we need to show $H^0(A_0, \mathcal{L}\otimes \mathcal{O}_{A_0}(-b))=0$, where $A_0$ is any smooth irreducible member in $|\mathcal{O}_X(a)|$. As explained in the Corollary 1.8, it suffices to show $H^1(X, \mathcal{L}(-a-b))=0$. Then there are two different situations: 1) the K3 surfaces $X$ do not have -2 divisors; 2) the K3 surfaces $X$ have -2 divisors.  In the first situation, it is very easy to check $H^1(X, \mathcal{L}(-a-b))=0$. (Cf. Lemma 2.8) In the second situation,  we explicitly compute the closed cone of curves $\overline{NE(X)}$ and the nef cone $Nef(X)$ and then use the information about these cones to check $H^1(X, \mathcal{L}(-a-b))=0$.
\vspace{2mm}
\subsection{$X$ doesn't have -2 divisors}

\subsubsection{Vanishing of the first cohomology group of line bundles on a K3 surface without -2 divisors}
\begin{proposition}
 ([10, Cor. 3.2]).  Let $\Sigma $ be a complete linear system on a K3 surface. Then $\Sigma$ has no base points outside its fixed components.

\end{proposition}

\begin{proposition}
 ([6, Prop. 2.3]) Let $|D|\neq \emptyset$ be a complete linear system without fixed components on a K3 surface such that $D^2=0$. Then every member of $|D|$ can be written as a sum $E_1+E_2+...+E_k,$ where $E_i\in|E|$ for $i=1,...,k$ and $E$ is a smooth curve of genus 1. 
 
 In other words, $|D|$ is a multiple $k$ of an elliptic pencil. 
 
 In particular, if $D$ is part of a basis of $Pic\; X$, then the generic member of $|D|$ is smooth and irreducible.

\end{proposition}

\begin{lemma}
Let $X$ be a K3 surface without -2 divisors. Let $D$ be a divisor on $X$. Then $H^1(X,\mathcal{O}_X(D))=0$ if and only if the following two conditions are satisfied

(i) $D^2\geq -4$

(ii) $\nexists$ a smooth elliptic curve $E$ on $X$ and an integer $k$, $\ni D\sim k E$ and $|k|>1$.

In particular, if $D$ is part of a basis of $Pic\; X$, then (ii) is automatically true.
\end{lemma}

\begin{proof}
By Riemann-Roch,  $H^1(X,\mathcal{O}_X(D))=0$ easily implies  $D^2\geq -4$. $H^1(X,\mathcal{O}_X(D))=0$ also implies (ii). Otherwise, $\exists$ a smooth elliptic curve $E$ on $X$ and an integer $k$, $\ni D\sim k E$ and $|k|>1$. We may assume $k$ is positive, then we have an exact sequence of cohomology groups $0\rightarrow H^0(X,\mathcal{O}_X(-D))\rightarrow H^0(X, \mathcal{O}_X)\rightarrow H^0(kE,\mathcal{O}_{kE})\rightarrow H^1(X,\mathcal{O}_X(-D))\rightarrow 0$. But $dim\;H^0(X, \mathcal{O}_X)=1$ and $dim\; H^0(kE,\mathcal{O}_{kE})=k>1$, so $H^1(X,\mathcal{O}_X(-D))\neq 0$ and hence $H^1(X,\mathcal{O}_X(D))\neq 0$ by Serre duality, contradiction.

On the other hand, suppose both (i) and (ii) are satisfied. Firstly, if $D^2>0$, by R-R, either $|D|$ or $|-D|$ is non-empty. We may assume $D$ is effective. Every irreducible curve $C$ on $X$ has non-negative self-intersection, so the linear system $|C|$ has no fixed components, and hence $|C|$ is base point free by Proposition 2.6. Therefore, every irreducible curve on $X$ is a nef divisor, and hence every effective divisor on $X$ is nef. So $D$ is nef and big, then $H^1(X,\mathcal{O}_X(D))=0$ by Kawamata-Viehweg vanishing Theorem. Secondly, if $D^2=0$, again we may assume $D$ is effective. Clearly $|D|$ doesn't have fixed components, then by Proposition 2.7 and (ii) $|D|$ is actually an elliptic pencil. Therefore, $|D|$ contains an irreducible member, and hence $H^1(X,\mathcal{O}_X(D))=0$. Lastly, if $D^2=-4$,  then both $|D|$ and $|-D|$ are empty, $H^1(X,\mathcal{O}_X(D))=0$ by R-R.

\end{proof}

\vspace{3mm}

\subsubsection{Curves on $X$ deformed with CICY threefolds}

\begin{theorem}
Let $d\geq 1$ and $g\geq 0$ be integers. Then in any of the following cases the general Calabi-Yau complete intersection threefold $Y$ of a particular type contains an isolated, smooth curve of degree $d$ and genus $g$:

\vspace{1mm}

(a) $Y=(5)\subset \mathbb{P}^4:  g=23$ and $d=18$; $g=24$ and $d=19$; $g=26$ and $d=20$; $g=27$ and $d=20$; $g=29$ and $d=21$.

(b)  $Y=(4,2)\subset \mathbb{P}^5:$ $g=16$ and $17\le d\le 19$; $g=17$ and $17\le d\le 20$; $g=18$ and $d=20$; $g=19$ and $d=18,20 \; and\; 21$; $g=20$ and $d=19, 20$; $g=21$ and $d=20, 21$; $g=22$ and $d=20, 21$; $g=23$ and $d=20, 22$; $g=25$ and $d=21, 22$; $g=26$ and $d=22$; $g=27$ and $d=22$; $g=29$ and $d=23$.

(c) $Y=(3,3)\subset \mathbb{P}^5: g=8$ and $d=12$.

(d) $Y=(2,2,3 )\subset \mathbb{P}^6: g=11$ and $d=16$.

(e) $Y=(2,2,2,2 )\subset \mathbb{P}^7:$ $g=4$ and $d=9$; $g=5$ and $d=10, 11$.

\end{theorem}

\begin{proof}

Case (a) g=23 and d=18:

By Theorem 2.3, there exists a $K3$ surface $X$ of degree 6 in $\mathbb{P}^4$ with $Pic\; X\cong \mathbb{Z}H\oplus \mathbb{Z}C$, where $H$ is the hyperplane section of $X$ and $C$ is a smooth irreducible curve of degree 18 and genus 23. Clearly, $X$ is actually a complete intersection K3 surface of type (2,3) in $\mathbb{P}^4$. $\mathcal{L}$ is defined to be the line bundle $\mathcal{O}_X(C)$.

Using notations introduced above (also the same notations as in [5, Table 1 in \S 6] except that [5] doesn't use $a$ and $b$).  : $r=4, \mu =4$ (cf. [5, Table 1 in \S 6]), $a_1=3, a_2=2, b_1=5, l=36, g=23, d=18, a=2$ and $b=3$.

There are two conditions in [5, Prop 7.2]:

\vspace{1mm}

(7.1)                                        \centerline{ $d\leq 2a_{r-2}(\mu-1)$ or $da_{r-2}>a_{r-2}^2(\mu-1)+g$}

\vspace{1mm}

(7.2)   \[  a_{r-2}(2a_{r-3}-a_{r-2})(\mu-1)\geq \left\{ \begin{array}{ll}
         g+2 & \mbox{if $a_{r-3}\neq a_{r-2}$};\\
        g+1 & \mbox{if $a_{r-3}=a_{r-2}$}.\end{array} \right. \]

\vspace{1mm}

The condition (7.1) is satisfied since $18\cdot 2>2^2\cdot(4-1)+23$.
The trouble is that condition (7.2) is not satisfied (notice that in the language of Lemma 1.6, condition (7.2) is exactly $n\le N$). However, by looking closely at the proof of [5, Prop 7.2], the condition (7.2) is only used to prove the following two statements: 

Statement 1). $l\geq g+2$, where $l$ is the number of nodes on a general quintic threefold containing $X$ as before.

Statement 2). For general $\alpha_{ij}$, the set of nodes $S$ imposes independent conditions on $|\mathcal{L}|$, $(\star)$.

Therefore, in order to get the conclusion of [5, Prop. 7.2] we only need to show statements 1) and 2).

$l=36>25=g+2$. Statement 1) is proved.

As in \S 1, we let $n=h^0(X, \mathcal{L})=24$.  Next we are going to use Corollary 1.8 to show $(\star)$. 

Notice that in \S 1, we assume throughout $H^0(X,\mathcal{L}(-a))=0$. So we need to show that in current situation we do have $H^0(X,\mathcal{L}(-2))=0$. (Actually, the following proof for $H^0(X,\mathcal{L}(-a))=0$ can be found in the proof of [5, Prop. 7.2]. For the reader's convenience, we repeat it here.)  By [6, Prop. 1.3] $h^1(C', \mathcal{O}_{C'}(a_{r-2}))=0$ for all $C'\in |\mathcal{L}|$ if and only if $$d\leq 2a_{r-2}(\mu-1) \;\text{or}\; da_{r-2}>a_{r-2}^2(\mu-1)+g,$$ which is condition (7.1). Next We note from the cohomology of $$0\longrightarrow \mathcal{L}^{\vee}\longrightarrow \mathcal{O}_X\longrightarrow \mathcal{O}_C\longrightarrow 0$$ twisted by $\mathcal{O}_X(a_{r-2})$, Kodaira vanishing and Serre duality, that $$h^0(X, \mathcal{L}\otimes\mathcal{O}_X(-a_{r-2}))=h^1(\mathcal{O}_C(a_{r-2})),$$ so that also $h^0(X, \mathcal{L}\otimes\mathcal{O}_X(-a_{r-2}))=0$ if condition (7.1) holds, as we have just seen.

 Now as in \S 1 we fix any smooth irreducible $A_0\in |\mathcal{O}_X(2)|$. First of all, condition i) is satisfied since 
$l=36$ and $dim|\mathcal{O}_{A_0}(3)|=23$. In order to show condition ii), we only need to prove $H^1(X,\mathcal{L}(-5))=0$. However, using some softwares (e.g. Mathematica) it is very easy to check that $X$ has no -2 divisors. For example, the following picture shows us how to do this by using Mathematica. We use $(x,y)$ to represent a divisor $xH+yC$.

\vspace{2mm}

\includegraphics{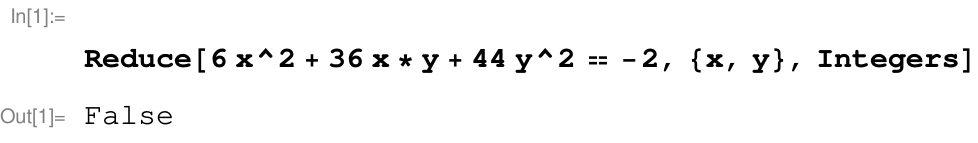}

\vspace{2mm}

By Lemma 2.8, $H^1(X,\mathcal{O}_X(5H-C))=0$ since $(5H-C)^2=14$. By Serre duality, $H^1(X,\mathcal{L}(-5))\cong H^1(X,\mathcal{O}_X(5H-C))^{\vee}=0$. Therefore, condition ii) in Corollary 1.8 is also true. So $(A_0, \mathcal{O}_X(3))$ can impose independent conditions on $|\mathcal{L}|$. Then by Bertini's Theorem, for general $\alpha_{ij}$, the set of nodes $S$ imposes independent conditions on $|\mathcal{L}|$. Therefore, we have the conclusion of [5, Prop7.2], which says the conditions (A1)-(A7)are satisfied. Then by Theorem 0.1, the general quintic threefold in $\mathbb{P}^4$ contains an isolated, smooth curve of degree 18 and genus 23.

\vspace{2 mm}
 The proofs for all the other cases are similar to that for the case (a) $g=23$ and $d=18$. However, we need to specify what the complete intersection type of K3 surface is for each case. The information needed is listed in Table 2 in the Appendix B. (All K3 surfaces $X$ in this table have no -2 divisors.)

 \end{proof}

\subsection{X has -2 divisors}
\subsubsection{Nef cone of K3 surfaces with -2 divisors}

If a K3 surface $X$ does not have -2 divisors, it is very easy to compute  the closed cone of curves $\overline{NE(X)}$ and hence the nef cone $Nef(X)$. (Cf. [8, Corollary 2.3]). If there are -2 divisors in $Pic\; X$, it could be difficult to compute the nef cone of $X$ in general. However, if the Picard number of $X$ is 2, it is not hard to do so. In this subsection we will assume throughout that the K3 surface $X$ has a -2 divisor and its Picard number is 2.

\begin{lemma}
Let X be a smooth projective K3 surface with Picard number 2. Then $\exists H, C\in Pic\;X,\ni Pic\;X=\mathbb{Z}H\oplus \mathbb{Z}C$, $H$ is an ample divisor, $H.C>0,$ and $C^2>0$.
\end{lemma}

\begin{proof}
Choose any basis for $Pic\; X$, say $Pic\; X=\mathbb{Z}A\oplus\mathbb{Z}B$. Because $X$ is projective, suppose $H=aA+bB$ is an ample class, and we may assume integers a and b are coprime. Then $\exists D\in Pic\;X,\ni Pic\;X=\mathbb{Z}H\oplus\mathbb{Z}D $. Let $C=nH+D$. Obviously, if $n$ is sufficiently large, we have $H.C>0,$ and $C^2>0$.
\end{proof}

Let $H,C$ be as in Lemma 2.2. Let $h=H^2, d=H.C$, and $c=C^2$. By [8, Theorem 2], we know the two boundary rays of $\overline{NE(X)}$ are spanned by: i) either two smooth rational curves; ii) or one smooth rational curve and one rational curve of self-intersection 0.

We will use $(x,y)$ to denote the $\mathbb{R}-$divisor $xH+yC$. Let's look at the following picture.

\includegraphics{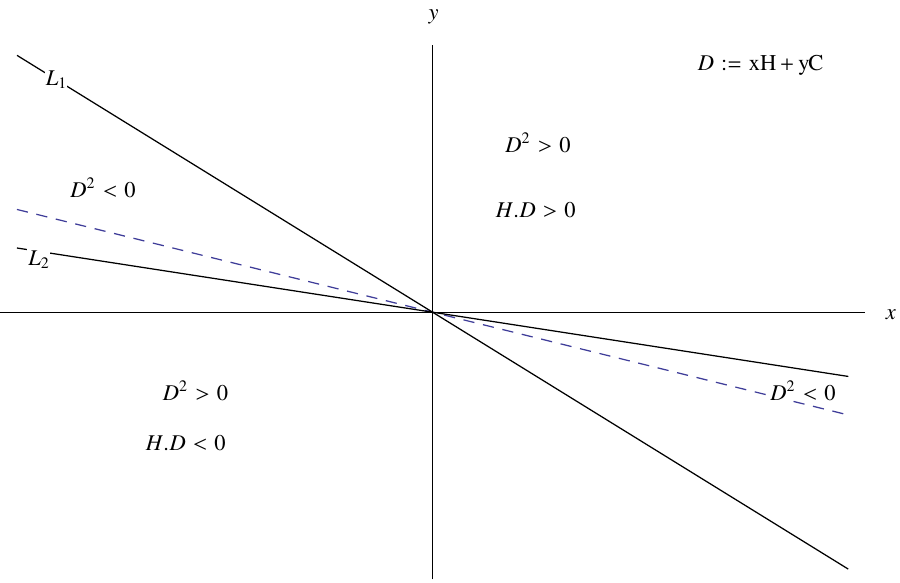}

The two lines $L_1$ and $L_2$ consist of  $\mathbb{R}-$divisors with self-intersection 0. Their slopes are $-\frac{d+\sqrt{d^2-hc}}{c}$ and $-\frac{d-\sqrt{d^2-hc}}{c}$. The dashed line is defined by the set of $\mathbb{R}-$divisors orthogonal to $H$. So its slope is $-\frac{h}{d}$. Notice that we always have $$0<\frac{d-\sqrt{d^2-hc}}{c}<\frac{h}{d}<\frac{d+\sqrt{d^2-hc}}{c}.$$

\begin{lemma}
Let $D=x_0H+y_0C$ be a -2 divisor, where $x_0$ and $y_0$ are integers. Then $D$ represents a smooth rational curve and hence spans a boundary ray of $\overline{NE(X)}$ if and only if the following two conditions are satisfied:

i) $D$ is effective, which implies that $D$ is on the right side of the dashed line,

ii) $\nexists$ integral -2 divisor $E=x_1H+y_1C$, $\ni$ the ratio $\frac{y_1}{x_1}$ is between $\frac{y_0}{x_0}$ and $-\frac{h}{d}$.
\end{lemma}

\begin{proof}
Suppose $D$ represents a smooth rational curve. Obviously i) is satisfied.  If $\exists$ integral -2 divisor $E=x_1H+y_1C$, $\ni$ the ratio $\frac{y_1}{x_1}$ is between $\frac{y_0}{x_0}$ and $-\frac{h}{d}$. Then $E$ is effective and hence $E\in \overline{NE(X)}$, which means the ray spanned by $D$ is not a boundary ray, contradiction. So ii) is satisfied.

Suppose both i) and ii) are satisfied. Then $D\in \overline{NE(X)}$. By ii) and [8, Theorem 2], $D$ spans a boundary ray of $\overline{NE(X)}$ and $D$ represents a smooth rational curve.

\end{proof}

Now let's analysis the slopes of the lines spanned by integral -2 divisors. Suppose $D=xH+yC$ is an integral -2 divisor. By definition, $hx^2+2dxy+cy^2=-2$. Then $\frac{y}{x}=-\frac{d\pm \sqrt{d^2-c(h+\frac{2}{x^2})} }{c}$. When $x$ goes to infinity, the line spanned by $D$ approaches the red lines in the picture above. Therefore, in order to find $D$ satisfying both i) and ii) in Lemma 2.11, we only need to find integral -2 divisors $D$ with $x$ coordinate ``small'', which is pretty easy with the help of some softwares(e.g. Mathematica).

Let's do an example in the following. Suppose $X$ is a smooth projective K3 surface with $Pic\;X=\mathbb{Z}H\oplus \mathbb{Z}C$, $H$ is ample, $h=H^2=6$, $d=H.C=19$, and $c=C^2=48$.

\vspace{2mm}
First step: Use Mathematica to determine all -2 divisors satisfying i) in Lemma 2.11. (Notice that no solutions means there are no -2 divisors at all.)

\vspace{5mm}

\includegraphics{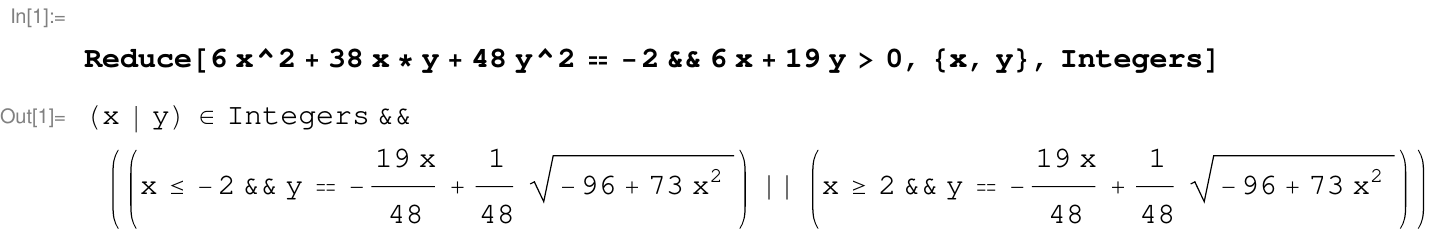}

\vspace{2mm}
From the result of the first step, in order to find integral -2 divisors $D$ satisfying both i) and ii) in Lemma 2.11, we only need to find the integral -2 divisors $D$ with $x$ as small as possible.

Second step: find -2 divisors with small $x$ coordinates:

\vspace{2mm}

\includegraphics{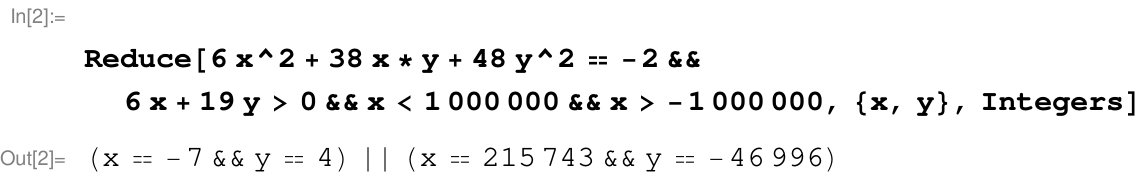}

\vspace{3mm}
Therefore, $\overline{NE(X)}$ is the closed cone spanned by $-7H+4C$ and $215743H-46996C$. Because the nef cone is the dual of $\overline{NE(X)}$, $Nef(X)$ is spanned by  $-59H+34C$ and $1843309H-401534C$.

\vspace{2mm}

\subsubsection{Curves on $X$ deformed with CICY threefolds}

\vspace{5mm}

\begin{theorem}
 Let $d\geq 1$ and $g\geq 0$ be integers. Then in any of the following cases the general Calabi-Yau complete intersection threefold $Y$ of a particular type contains an isolated, smooth curve of degree $d$ and genus $g$:

\vspace{5mm}

(a) $Y=(5)\subset \mathbb{P}^4:  g=23$ and $d=19$; $g=24$ and $d=20$; $g=25$ and $d=19, 20$.

(b)  $Y=(4,2)\subset \mathbb{P}^5:$ $g=16$ and $d=20$; $g=18$ and $d=18, 19$; $g=19$ and $d=19$; $g=20$ and $d=21$; $g=21$ and $d=19$; $g=23$ and $d=21$;  $g=24$ and $d=21, 22$; $g=27$ and $d=23$; $g=28$ and $d=23$.

(c) $Y=(2,2,3 )\subset \mathbb{P}^6: g=4$ and $d=8$.

(d) $Y=(2,2,2,2 )\subset \mathbb{P}^7:$ $g=4$ and $d=10$; $g=6$ and $d=11$.

\end{theorem}

\begin{proof}
Case (a) g=25 and d=19:

By Theorem 2.3, there exists a $K3$ surface $X$ of degree 6 in $\mathbb{P}^4$ with $Pic\; X\cong \mathbb{Z}H\oplus \mathbb{Z}C$, where $H$ is the hyperplane section of $X$ and $C$ is a smooth irreducible curve of degree 19 and genus 25. Clearly, $X$ is actually a complete intersection K3 surface of type (2,3) in $\mathbb{P}^4$. $\mathcal{L}$ is defined to be the line bundle $\mathcal{O}_X(C)$.

Using notations introduced above: $r=4, \mu =4, a_1=3, a_2=2, b_1=5, l=36, g=25, d=19, a=2$ and $b=3$.

Similar to the proof of Theorem 2.9, the condition (7.1) is satisfied since $19\cdot 2>2^2\cdot(4-1)+25$, and the condition (7.2) is not satisfied. As before we only need to prove the following two statements to  get the conclusion of [5, Prop. 7.2]:

Statement 1). $l\geq g+2$, where $l$ is the number of nodes on a general quintic threefold containing $X$ as before.

Statement 2). For general $\alpha_{ij}$, the set of nodes $S$ imposes independent conditions on $|\mathcal{L}|$, $(\star)$.

$l=36>27=g+2$. Statement 1) is proved.

As in \S 1, we let $n=h^0(X, \mathcal{L})=26$.  Next we are going to use Corollary 1.8 to show $(\star)$. 

Notice that in \S 1, we assume throughout $H^0(X,\mathcal{L}(-a))=0$. So we need to show that in current situation we do have $H^0(X,\mathcal{L}(-a))=0$. Actually, we have proved that the condition (7.1) implies $H^0(X,\mathcal{L}(-a))=0$ in the proof of Theorem 2.9.

 Now as in \S 1 we fix any smooth irreducible $A_0\in |\mathcal{O}_X(2)|$. First of all, condition i) in Cor. 1.8 is satisfied since 
$l=36$ and $dim|\mathcal{O}_{A_0}(3)|=23$. In order to show condition ii), we only need to prove $H^1(X,\mathcal{L}(-5))=0$. Using the method introduced above we can find the nef cone $Nef(X)$ explicitly. $Nef(X)$ is spanned by  $-59H+34C$ and $1843309H-401534C$. Therefore $5H-C$ is a nef divisor. $ (5H-C)^2=8>0$, so $5H-C$ is nef and big and hence $H^1(X,\mathcal{L}(-5))\cong H^1(X,\mathcal{O}_X(5H-C))^{\vee}=0$ by Kawamata-Viehweg vanishing and Serre duality.  Therefore, condition ii) in Corollary 1.8 is also true. So $(A_0, \mathcal{O}_X(3))$ can impose independent conditions on $|\mathcal{L}|$. Then by Bertini's Theorem, for general $\alpha_{ij}$, the set of nodes $S$ imposes independent conditions on $|\mathcal{L}|$. Therefore, we have the conclusion of [5, Prop7.2], which says the conditions (A1)-(A7)are satisfied. Then by Theorem 0.1, the general quintic threefold in $\mathbb{P}^4$ contains an isolated, smooth curve of degree 19 and genus 25.

\vspace{5 mm}
 The proofs for all the other cases are similar to that for the case (a) $g=25$ and $d=19$. However, as before we need to specify what the complete intersection type of K3 surface is for each case. The information needed is listed in Table 3 in the Appendix B. (All K3 surfaces $X$ in this table have -2 divisors.)

\vspace{5 mm}

\end{proof}

\begin{center}
\section{Curves on rational surfaces in nodal Calabi-Yau threefolds}
\end{center}

\begin{theorem}
Let $d\geq 1$ and $g\geq 0$ be integers. Then in any of the following cases the general Calabi-Yau complete intersection threefold $Y$ of a particular type contains an isolated, smooth curve of degree $d$ and genus $g$:

(i) $Y=(3,3)\subset \mathbb{P}^5: g=3$ and $d=6$.

(ii) $Y=(2,4)\subset \mathbb{P}^5: g=3$ and $d=6$.

\end{theorem}

\begin{remark}
In order to prove this theorem, we choose appropriate complete intersection rational surfaces $X$ and line bundles $\mathcal{L}$ on $X$, check conditions (A1)-(A7) in [5] and then apply Theorem 0.1. For case (i) we choose $X=(3)\subset\mathbb{P}^3, \mathcal{L}=\mathcal{O}_X(H+l)$(which will be explained in the proof); for case (ii) we choose $X=(2)\subset\mathbb{P}^3, \mathcal{L}=(2,4)\in \mathbb{Z}\times\mathbb{Z}\cong Pic(X)$.(cf. [4, Ch.V \S 4] )

The proof for case (ii) is completely similar to that for case (i), so we will only prove case (i).
\end{remark}

Notations as in [5]: \\$\mathcal{P}=\mathbb{P}^5$ , $\mathcal{E}=\mathcal{M}_1\oplus\mathcal{M}_2=\mathcal{O}_{\mathbb{P}^5}(3)\oplus\mathcal{O}_{\mathbb{P}^5}(3)$, \\ $X$ is a smooth cubic surface in $\mathbb{P}^3\subset\mathbb{P}^5$ defined by a cubic homogeneous form $F=0$ (for  simplicity, we will assume $F=x_0^3+x_1^3+x_2^3+x_3^3$).\\ $Z=\{$\~{F}$=x_0^3+x_1^3+x_2^3+x_3^3+x_4^3+x_5^3=0\}\subset\mathbb{P}^5$ smooth cubic hypersurface containing $X$, \\ $Y=\{$\~{F}$=x_4G+x_5P=0 \}\subset Z$ where $G, P\in H^0(\mathbb{P}^5, \mathcal{O}_{\mathbb{P}^5}(2))$.
\\  (*) $X\cong$ Blow-up of $\mathbb{P}^2$ at six points. $H=3l-\sum_{i=1}^{6} e_i$, where $l$ is pullback of $\mathcal{O}_{\mathbb{P}^2}(1)$, $H$ is the hyperplane section of $X$. Let $\mathcal{L}$ be the line bundle $\mathcal{O}_X(H+l)$ on $X$.

\begin{lemma}
Suppose $A$ is a smooth curve of genus 4 and degree 6 in $\mathbb{P}^3$, and $A$ is complete intersection of a smooth quadric surface and a smooth cubic surface X. Then $H^0(A, \mathcal{O}_A(H_A-l_A))=0$,  where $H_A$ and $l_A$ are restriction of $H$ and $l$ to $A$($H$ and $l$ which are line bundles on $X$ are defined as above) .
\end{lemma}

\begin{proof}
Consider short exact sequence of sheaves on $X$ : $$0\rightarrow \mathcal{O}_{X}(-2H)\rightarrow \mathcal{O}_X\rightarrow \mathcal{O}_A\rightarrow 0$$ tensoring $\mathcal{O}_{X}(H-l)$ , we get $$0\rightarrow \mathcal{O}_{X}(-H-l)\rightarrow \mathcal{O}_X(H-l)\rightarrow \mathcal{O}_A(H_A-l_A)\rightarrow 0$$

Taking cohomology groups, we have $H^0(X,\mathcal{O}_{X}(H-l))\rightarrow H^0(A, \mathcal{O}_A(H_A-l_A))\rightarrow H^1(X, \mathcal{O}_X(-H-l))$ exact.  $H^0(X,\mathcal{O}_{X}(H-l))=0$ since $(H-l).H=0$ and $\mathcal{O}_X(H)\ncong\mathcal{O}_X(l)$,  $H^1(X, \mathcal{O}_X(-H-l))=0$ (by Kodaira vanishing), then $H^0(A, \mathcal{O}_A(H_A-l_A))=0$ .

\end{proof}

\begin{lemma}
There exist $G$, $P$ satisfying the following conditions: \\ 1) $Y$ only has nodal singularities which are in $X$ \\ 2) $\{ G=x_4=x_5=0\}\subset \mathbb{P}^3$ is a smooth quadric surface. $A:=X\cap \{G=0 \}$ is a smooth curve of genus 4 and degree 6 \\ 3) $S:= $the set of nodes of $Y=X\cap\{ G=P=0\}=$12 distinct points \\ 4) $S$ imposes independent conditions on $\mid\mathcal{L}\mid$
\end{lemma}

\begin{proof} 
By Bertini's Theorem, for general $(G, P)\in H^0(\mathbb{P}^5,\mathcal{O}_{\mathbb{P}^5}(2))\times H^0(\mathbb{P}^5,\mathcal{O}_{\mathbb{P}^5}(2)) $ , conditions 1)-3) are satisfied. That means $\exists $ Zariski open dense subset $U$ of $H^0(\mathbb{P}^5,\mathcal{O}_{\mathbb{P}^5}(2))\times H^0(\mathbb{P}^5,\mathcal{O}_{\mathbb{P}^5}(2))$ such that $\forall (G, P)\in U$ , conditions 1)-3) are satisfied. Now fix $(G, P)\in U$ . 

By Lemma 3.3, the condition ii) in Corollary 1.8 is satisfied. 

The conditions i) in corollary 1.8 are satisfied since using the same notations as in \S 1 $a=b=2, l=12$ and  $n=h^0(X, \mathcal{L})=9$.

$H^0(X,\mathcal{O}_{X}(\mathcal{L}-2H))\cong H^0(X,\mathcal{O}_{X}(l-H))=0$ since $(l-H).H=0$ and $\mathcal{O}_X(H)\ncong\mathcal{O}_X(l)$. Therefore, the assumption (1) in the \S 1 is also satisfied. So by corollary 1.8, $(A, \mathcal{O}_X(2H))$ can imposes independent conditions on $|\mathcal{L}|$.

 Notice that the natural restriction map $\rho : H^0(\mathbb{P}^5, \mathcal{O}_{\mathbb{P}^5}(2))\rightarrow H^0(A, \mathcal{O}_A(2))$ is surjective, so for general $P'\in H^0(\mathbb{P}^5, \mathcal{O}_{\mathbb{P}^5}(2)), \rho (P')$ viewed as an element in $|2H_A|$ imposes independent conditions on $|\mathcal{L}|$. Therefore, $\exists P'\in H^0(\mathbb{P}^5, \mathcal{O}_{\mathbb{P}^5}(2)) $ such that the $(G, P')\in U$ and $\rho (P')$ imposes independent conditions on $|\mathcal{L}|$ .

\end{proof}

\begin{proof} $of$ $Theorem$ $3.1$ (case (i)):  Notations as in Lemma 3.4. We will check the conditions (A1)-(A7) listed in the introduction. \\ (A1): Trivial; \\ (A2) By definition  $Z$ is a smooth hypersurface in $\mathbb{P}^5$ and hence smooth along $X$. The only singularities of $Y$ which lie in $X$ are 12 nodes $\xi_1,..., \xi_{12}$ . Furthermore, by R-R for $X$ , $h^0(X, \mathcal{L})=9$ and hence  $$12\geq dim|\mathcal{L}|+2=10 ;$$ \\ (A3) and (A4): Trivial (Simply because $\mathcal{L}$ is very ample on $X$); \\ (A5): (A5) is equivalent to the condition (A5') : The set of nodes $S$ imposes independent conditions on $|\mathcal{L}|$ and the natural map $\gamma_{C}: H^0(C, \mathcal{N}_{X/Y}\otimes \mathcal{O}_{C})\rightarrow H^1(C, \mathcal{N}_{C/X})$ is an isomorphism for all $C\in |\mathcal{L}|$ . 

Let $C\in |\mathcal{L}|$ , we have exact sequence $$0\rightarrow \mathcal{O}_X\rightarrow \mathcal{L}\rightarrow \mathcal{N}_{C/X}\rightarrow 0$$  taking cohomology groups, we get exact sequence $$H^1(X, \mathcal{L})\rightarrow H^1(C, \mathcal{N}_{C/X})\rightarrow H^2(X, \mathcal{O}_X) $$   Then $H^1(C, \mathcal{N}_{C/X})=0$ because $H^1(X, \mathcal{L})=0$ and $H^2(X, \mathcal{O}_{X})=0$ . By Serre duality, $ H^0(C, \mathcal{N}_{X/Y}\otimes \mathcal{O}_{C})\cong H^1(C, \mathcal{N}_{C/X})$ (cf. [5, (4.2)]) , so we have $ H^0(C, \mathcal{N}_{X/Y}\otimes \mathcal{O}_{C})= H^1(C, \mathcal{N}_{C/X})=0$ and hence $\gamma_{C}$ is an isomorphism for trivial reason. By condition 4) in Lemma 3.4, $S$ imposes independent conditions on $|\mathcal{L}|$ . Therefore, (A5') is satisfied; \\ (A6): $\forall C\in |\mathcal{L}|$ we have exact sequence (cf. [5, (3.9)]) $$0\rightarrow \mathcal{N}_{C/X}\rightarrow \mathcal{N}_{C/\mathbb{P}^5}\rightarrow \mathcal{N}_{X/\mathbb{P}^5}\otimes \mathcal{O}_{C}\rightarrow 0$$ 

Again taking cohomology groups, we get exact sequence $$H^1(C, \mathcal{N}_{C/X})\rightarrow H^1(C, \mathcal{N}_{C/\mathbb{P}^5})\rightarrow H^1(C, \mathcal{N}_{X/\mathbb{P}^5}\otimes \mathcal{O}_{C}) $$ \\ $H^1(C, \mathcal{N}_{X/\mathbb{P}^5}\otimes \mathcal{O}_{C})\cong H^1(C,\mathcal{O}_{C}(3)\oplus \mathcal{O}_{C}(1)\oplus \mathcal{O}_{C}(1))$ , and we have seen $ H^1(C, \mathcal{N}_{C/X})=0$ . In order to show $H^1(C, \mathcal{N}_{C/\mathbb{P}^5})=0$, we just need to show $H^1(C, \mathcal{O}_{C}(1))=H^1(C, \mathcal{O}_{C}(3))=0$ . But by considering exact sequences $0\rightarrow \mathcal{O}_{X}(H-\mathcal{L})\rightarrow \mathcal{O}_X(1)\rightarrow \mathcal{O}_{C}(1)\rightarrow 0$ and $0\rightarrow \mathcal{O}_{X}(3H-\mathcal{L})\rightarrow \mathcal{O}_X(3)\rightarrow \mathcal{O}_{C}(3)\rightarrow 0$ , it's easy to show that $H^1(C, \mathcal{O}_{C}(1))=H^1(C, \mathcal{O}_{C}(3))=0$ ; \\ (A7): We want to show the image of the natural restriction map $$H^0(\mathbb{P}^5 , \mathcal{O}_{\mathbb{P}^5}(3))\rightarrow H^0(S, \mathcal{O}_{\mathbb{P}^5}(3)\otimes \mathcal{O}_{S})$$ has codimension one. Notice that the natural restriction map $$H^0(\mathbb{P}^5 , \mathcal{O}_{\mathbb{P}^5}(3))\rightarrow H^0(A, \mathcal{O}_{A}(3))$$ is surjective, where $A=X\cap\{ G=0\}$ as in Lemma 3.4. Therefore, we just need to show the image of the natural restriction map $$H^0(A , \mathcal{O}_{A}(3))\rightarrow H^0(S, \mathcal{O}_{S}(3))$$ has codimension one. Clearly we have the following exact sequence: $$0\rightarrow \mathcal{O}_{A}(1)\rightarrow \mathcal{O}_A(3)\rightarrow\mathcal{O}_S(3)\rightarrow 0$$  Taking cohomology groups we get exact sequence $$H^0(A , \mathcal{O}_{A}(3))\rightarrow H^0(S, \mathcal{O}_{S}(3))\rightarrow H^1(A , \mathcal{O}_{A}(1))\rightarrow H^1(A, \mathcal{O}_{A}(3))$$ but $H^1(A , \mathcal{O}_{A}(1))\cong H^0(A, \omega_A\otimes \mathcal{O}_{A}(-1))^{\vee}=H^0(A, \mathcal{O}_{A})^{\vee}\cong \mathbb{C}$ and $H^1(A , \mathcal{O}_{A}(3))\cong H^0(A, \omega_A\otimes \mathcal{O}_{A}(-3))^{\vee}=H^0(A, \mathcal{O}_{A}(-2))^{\vee}=0$ . So (A7) is satisfied.

\end{proof}
\newpage
\appendix

\section{Method II}

As mentioned in Page 6, we present another method to check the condition ($\star$) in this appendix. We follow notations and conventions used in \S 1.

\begin{lemma}
Define $Z:=\{(D, B, C)\in A_0^{(n)}\times |\mathcal{O}_{A_0}(b)|\times|\mathcal{L}\otimes\mathcal{O}_{A_0}|  | D\subset B,D\subset C\}$. Let $p: Z\longrightarrow |\mathcal{O}_{A_0}(b)|$ be the projection map. We assume $H^1(X, \mathcal{O}_{X}(b-a))=0$. Then $(A_0, \mathcal{O}_{A_0}(b))$ can impose independent conditions on $|\mathcal{L}|$ if $p$ is not surjective.
\end{lemma}

\begin{proof}
Because $H^0(X, \mathcal{L}(-a))=0$, so the natural restriction map $H^0(X, \mathcal{L})\longrightarrow H^0(A_0, \mathcal{L}\otimes\mathcal{O}_{A_0})$ is injective. Then any element $C\in |\mathcal{L}|$ corresponds to a divisor $C\cap A_0\in |\mathcal{L}\otimes\mathcal{O}_{A_0}|$. 

Suppose $p$ is not surjective. Then $|\mathcal{O}_{A_0}(b)|\setminus p(Z)$ is open dense in $|\mathcal{O}_{A_0}(b)|$. Because $H^1(X, \mathcal{O}_{X}(b-a))=0$, the natural restriction map $H^0(X, \mathcal{O}_X(b))\longrightarrow H^0(A_0, \mathcal{O}_{A_0}(b))$ is surjective. So there exists $B\in |\mathcal{O}_X(b)|$ such that $B\cap A_0$ is a set of $l$ distinct points and $B\cap A_0 \in |\mathcal{O}_{A_0}(b)|\setminus p(Z)$. But $B\cap A_0\in |\mathcal{O}_{A_0}(b)|\setminus p(Z)$ precisely means that no member of $|\mathcal{L}\otimes\mathcal{O}_{A_0}|$ can contain $n$ points of $B\cap A_0$. Consequently, at most $n-1$ points of $B\cap A_0$ lies on an element of $|\mathcal{L}|$.
\end{proof}

\begin{remark}
If the restriction map $H^0(X, \mathcal{L})\longrightarrow H^0(A_0, \mathcal{L}\otimes\mathcal{O}_{A_0})$ is an isomorphism, then $(A_0, \mathcal{O}_{X}(b))$ can impose independent conditions on $|\mathcal{L}|$ if and only if $p$ is not surjective.

\end{remark}

 Next consider the following diagram:

\xymatrix{
\phi^{-1}(\mathcal{O}_{A_0}(b)\otimes\mathcal{L}^{\vee}) \ar@{^{(}->}[r] &A_0^{(l-n)}\times A_0^{(m-n)}\ar[r]^-\phi  & Pic^{l-m}A_0\\
Z \ar[u]^q\ar[r]^p &|\mathcal{O}_{A_0}(b)|}

\vspace{2 mm}

where $m=c_1(\mathcal{O}_X(a))\cdot c_1(\mathcal{L}), \phi (D_1,D_2)=\mathcal{O}_{A}(D_1-D_2)$, $q(D,B,C)=(B-D, C-D)$ and $p$ is the natural projection map. In order to show that $p$ is not surjective, it suffices to show that $dimZ<dim|\mathcal{O}_{A_0}(b)|$ . Roughly speaking, $Z$ has ``small" dimension if the fiber $\phi^{-1}(\mathcal{O}_{A_0}(b)\otimes\mathcal{L}^{\vee})$ has ``small" dimension and the fibers of the map $q$ has ``small" dimension. Finally, we need to do the following two steps:

\vspace{6 mm}
 $\mathbf{Step 1}:$ Show that the fiber $\phi^{-1}(\mathcal{O}_{A_0}(b)\otimes\mathcal{L}^{\vee})$ has minimum dimension,

\vspace{8 mm}
 $\mathbf{Step 2)}$:Analyze the map $q$ to show that general fiber of $q$ have small dimensions and the dimensions of special fibers are not too large.

The following general result provide us a way to do Step 1) in the special case $m-n=g_{A_0}-1$.

\begin{lemma}
Let $A$ be a smooth projective curve. Consider the difference map $\phi : A^{(r)}\times A^{(g_A-1)}\longrightarrow Pic^{r-g_A+1}A$, where $\phi (D_1, D_2)=\mathcal{O}_A(D_1-D_2)$, $g_A>1$ and $r>0$. $\mathcal{L}$ is a line bundle on $A$ of degree $r-g_A+1$. Then $dim\phi^{-1}(\mathcal{L})=r-1$ if and only if $H^0(A, \mathcal{L}^{\vee})=0$.
\end{lemma}
\begin{proof}
Define $f:A^{(r)}\times A^{(g_A-1)}\longrightarrow A^{(r+g_A-1)}$, $g:A^{(r+g_A-1)}\longrightarrow Pic^{r+g_A-1}A$, where $f$ is the addition map and $g(D)=\mathcal{O}_{A}(D), \forall D\in A^{(r+g_A-1)}$.  Let $\psi$ be the composite map $g\circ f$. Consider the following commutative diagram:

\xymatrix{
 A^{(r)}\times A^{(g_A-1)} \ar @{.>}[d]_-{\cong}^-{\tau} \ar[r]^-\alpha & Pic^{r}A\times Pic^{g_A-1}A\ar[d]_-{\cong}^-\beta \ar[r]^-\gamma  &Pic^{r-g_A+1}A\ar[d]_-{\cong}^-\kappa \\
A^{(r)}\times A^{(g_A-1)} \ar@{=}[d] \ar[r]^-\alpha&  Pic^{r}A\times Pic^{g_A-1}A\ar[r]^-\sigma   &Pic^{r+g_A-1}A\ar@{=}[d]  \\
 A^{(r)}\times A^{(g_A-1)} \ar[r]^-f&  A^{(r+g_A-1)} \ar[r]^-g&  Pic^{r+g_A-1}A  }

\vspace{8mm}

where $\alpha (D_1,D_2)=(\mathcal{O}_A(D_1), \mathcal{O}_A(D_2)), \gamma (\mathcal{G}_1, \mathcal{G}_2)=\mathcal{G}_1\otimes\mathcal{G}_2^{\vee}, \tau (D_1, D_2)=(D_1, D_2^{'})$ if $dim|\mathcal{O}_A(D_2)|=0$ and $D_2^{'}\in |\omega_A\otimes \mathcal{O}_A(-D_2)|$, $\beta (\mathcal{G}_1, \mathcal{G}_2)=(\mathcal{G}_1, \omega _A\otimes\mathcal{G}_2^{\vee})$, $\kappa (\mathcal{G})=\omega_A\otimes\mathcal{G}$, $\sigma (\mathcal{G}_1, \mathcal{G}_2)=\mathcal{G}_1\otimes\mathcal{G}_2$ and $\omega_A$ is the canonical sheaf of $A$.

Then $\tau$ is a birational map, $\beta$ and $\kappa$ are isomorphisms. $\phi=\gamma \circ\alpha$, $\psi=\sigma\circ\alpha=g\circ f$. Let $\mathcal{G}\in Pic^{r-g_A+1}A$, then $\gamma^{-1}(\mathcal{G})\cong \sigma^{-1}(\kappa(\mathcal{G}))$. $\forall (\mathcal{G}_1, \mathcal{G}_2)\in \gamma^{-1}(\mathcal{G})$, $dim\alpha^{-1}(\mathcal{G}_1, \mathcal{G}_2)=dim\alpha^{-1}(\beta(\mathcal{G}_1, \mathcal{G}_2))$ by Riemann-Roch. Therefore, $dim\phi^{-1}(\mathcal{G})=dim\psi^{-1}(\kappa(\mathcal{G}))$.

On the other hand, $f$ is a finite surjective map,  so $dim\psi^{-1}(\kappa(\mathcal{G}))=dim g^{-1}(\kappa(\mathcal{G}))=^{R.R.}dim|\omega_A\otimes (\kappa (\mathcal{G}))^{\vee}|+(r+g_A-1)+1-g_A=dim|\omega_A\otimes (\kappa (\mathcal{G}))^{\vee}|+r$.  Therefore, the fiber $\phi^{-1}(\mathcal{G})$ has minimum dimension, say $r-1$, if and only if $dim|\omega_A\otimes (\kappa (\mathcal{G}))^{\vee}|=-1$, i.e. $h^0(A,\omega_A\otimes (\kappa (\mathcal{G}))^{\vee})=0$. But by the definition of the map $\kappa$, $\omega_A\otimes (\kappa (\mathcal{G}))^{\vee}\cong \mathcal{G}^{\vee}$,  so we are done.

\end{proof}

\newpage

\section{Tables for the proof of Theorem 2.9 and Theorem 2.12}

$$\text{Table 2}$$
 \begin{tabular}{|p{2.5cm}|p{3cm}|p{3cm}| p{2.5cm}|c|}
 \hline 
 ($g$, $d$) &$Y$  &$X$  & $D:=(a+b)H-C$ 
 & $D^2$\\
 \hline 
 (23, 18) & $(5)$ & $(3,2)$  &$5H-C$&14 \\
 \hline 
 (24, 19) & $(5)$ & $(3,2)$  &$5H-C$&6\\
 \hline

 (26, 20) & $(5)$ & $(3,2)$  &$5H-C$ & 0\\
 \hline 
 (27, 20) & $(5)$ & $(3,2)$  &$5H-C$ & 2\\
 \hline 
 (29, 21) & $(5)$ & $(3,2)$  &$5H-C$ & -4\\
 \hline 
 (16, 17) & $(2,4)$ & $(2,2,2)$  &$4H-C$ & 22 \\
 \hline 
 (16, 18) & $(2,4)$ & $(2,2,2)$  &$4H-C$ & 14 \\
 \hline 
 (16, 19) & $(2,4)$ & $(2,2,2)$  &$4H-C$ & 6 \\
 \hline 
 
 (17, 17) & $(2,4)$ & $(2,2,2)$  &$4H-C$ & 24\\
 \hline 
 (17, 18) & $(2,4)$ & $(2,2,2)$  &$4H-C$ & 16\\
 \hline 
 (17, 19) & $(2,4)$ & $(2,2,2)$  &$4H-C$ & 8 \\
 \hline 
 (17, 20) & $(2,4)$ & $(2,2,2)$  &$4H-C$ & 0\\
 \hline 
 
 (18, 20) & $(2,4)$ & $(2,2,2)$  &$4H-C$ & 2 \\
 \hline 
 (19, 18) & $(2,4)$ & $(2,2,2)$  &$4H-C$ & 20\\
 \hline 
 
 (19, 20) & $(2,4)$ & $(2,2,2)$  &$4H-C$ & 4 \\
 \hline 
 (19, 21) & $(2,4)$ & $(2,2,2)$  &$4H-C$ & -4\\
 \hline 
 (20, 19) & $(2,4)$ & $(2,2,2)$  &$4H-C$ & 14\\
 \hline 
 (20, 20) & $(2,4)$ & $(2,2,2)$  &$4H-C$ & 6 \\
 \hline
 
  (21, 20) & $(2,4)$ & $(2,2,2)$  &$4H-C$ & 8\\
 \hline 
 (21, 21) & $(2,4)$ & $(2,2,2)$  &$4H-C$ & 0 \\
 \hline
 (22, 20) & $(2,4)$ & $(2,2,2)$  &$4H-C$ & 10\\
 \hline 
 (22, 21) & $(2,4)$ & $(2,2,2)$  &$4H-C$ & 2 \\
 \hline
 (23, 20) & $(2,4)$ & $(2,2,2)$  &$4H-C$ & 12\\
 \hline 
 
 (23, 22) & $(2,4)$ & $(2,2,2)$  &$4H-C$ & -4\\
 \hline

 (25, 21) & $(2,4)$ & $(2,2,2)$  &$4H-C$ & 8 \\
 \hline
 (25, 22) & $(2,4)$ & $(2,2,2)$  &$4H-C$ & 0\\
 \hline
 (26, 22) & $(2,4)$ & $(2,2,2)$  &$4H-C$ & 2 \\
 \hline
 (27, 22) & $(2,4)$ & $(2,2,2)$  &$4H-C$ & 4 \\
 \hline
 (29, 23) & $(2,4)$ & $(2,2,2)$  &$4H-C$ & 0\\
 \hline
 (8, 12) & $(3,3)$ & $(3,2,1)$  &$3H-C$ & -4\\
 \hline 
 
 (11, 16) & $(2,2,3)$ & $(2,2,2,1)$  &$3H-C$ & -4\\
 \hline
 (4, 9) & $(2,2,2,2)$ & $(2,2,2,1,1)$  &$2H-C$ & 2\\
 \hline
 
 (5, 10) & $(2,2,2,2)$ & $(2,2,2,1,1)$  &$2H-C$ & 0\\
 \hline
 (5, 11) & $(2,2,2,2)$ & $(2,2,2,1,1)$  &$2H-C$ & -4\\
 \hline

 \end{tabular}

\begin{table}
    \begin{adjustwidth}{-0.5in}{-.5in}  
        \begin{center}
$$\text{Table 3}$$

\begin{tabular}{|c|c|c| p{2.5cm}| p{2.5cm}| p{2cm}| c|p{3cm}|}
\hline 
($g$, $d$) &$Y$  &$X$  & $\overline{NE(X)}$ spanned by & $Nef(X)$ spanned by&$D:=(a+b)H-C$ 
& $D^2$& reason for $h^1(X,\mathcal{O}_X(D))=0$\\
\hline 

(23,19) &  $(5)$ &$(3,2)$   &$-34451H+22588C$, $827H-172C$  & $-339303H+222466C$, $8145H-1694C$ &$5H-C$& 4& $D$ nef and big\\
\hline

(24,20) &  $(5)$ &$(3,2)$   &$-3H+2C$, $192H-37C$  & $-16H+11C$, $1069H-206C$ &$5H-C$& -4& Both $|D|$ and $|-D|$ are empty, then use R-R \\
\hline
(25,19) &  $(5)$ &$(3,2)$   & $-7H+4C$, $215743H-46996C$  & $-59H+34C$, $1843309H-401534C$ &$5H-C$& 8& $D$ nef and big\\
\hline 
(25,20) &  $(5)$ &$(3,2)$   & $-11H+7C$, $5H-C$  & $-37H+58C$, $5H-26C$ &$5H-C$& -2& $|D|$ contains a smooth rational curve.\\
\hline 

(16,20) &  $(2,4)$ &$(2,2,2)$   & $-H+C$, $4H-C$  & $-5H+6C$, $25H-6C$ &$4H-C$& -2& $|D|$ contains a smooth rational curve.\\
\hline 

(18,18) &  $(2,4)$ &$(2,2,2)$   & $-147H+109C$, $3H-C$  & $-530H+393C$, $10H-3C$ &$4H-C$& 18& $D$ nef and big\\
\hline 
(18,19) &  $(2,4)$ &$(2,2,2)$   & $-6H+5C$, $32006H-9005C$  & $-56H+47C$, $301944H-84953C$ &$4H-C$& 10& $D$ nef and big\\
\hline 

(19,19) &  $(2,4)$ &$(2,2,2)$   & $-17H+13C$, $66233H-19237C$  & $-145H+111C$, $565895H-164361C$ &$4H-C$& 12& $D$ nef and big\\
\hline 

(20,21) &  $(2,4)$ &$(2,2,2)$   & $-425540H+366241C$, $4H-C$  & $-4980818H+4286741C$, $46H-11C$ &$4H-C$& -2& $|D|$ contains a smooth rational curve.\\
\hline 
(21,19) &  $(2,4)$ &$(2,2,2)$   & $-707H+449C$, $3H-C$  & $-4527H+2875C$, $17H-5C$ &$4H-C$& 16& $D$ nef and big\\
\hline

(23,21) &  $(2,4)$ &$(2,2,2)$   & $-13019H+9005C$, $19H-5C$  & $-122821H+84953C$, $179H-47C$ &$4H-C$& 4& $D$ nef and big\\
\hline

\end{tabular}

      \end{center}
    \end{adjustwidth}
\end{table}

\begin{table}
    \begin{adjustwidth}{-1in}{-.5in}  
        \begin{center}
$$\text{Table 3 (continued)}$$

\begin{tabular}{|c|c|c|p{2.5cm}| p{2.5cm}| p{2cm}| c|p{3cm}|}
\hline 
($g$, $d$) &$Y$  &$X$  & $\overline{NE(X)}$ spanned by & $Nef(X)$ spanned by&$D:=(a+b)H-C$ 
& $D^2$& reason for $h^1(X,\mathcal{O}_X(D))=0$\\
\hline

(24,21) &  $(2,4)$ &$(2,2,2)$   & $-29952H+19237C$, $48H-13C$  & $-255910H+164361C$, $410H-111C$ &$4H-C$& 6& $D$ nef and big\\
\hline
(24,22) &  $(2,4)$ &$(2,2,2)$   & $-984H+701C$, $4H-C$  & $-5299H+3775C$, $21H-5C$ &$4H-C$& -2& $|D|$ contains a smooth rational curve.\\
\hline 
(27,23) &  $(2,4)$ &$(2,2,2)$   & $-26405H+17077C$, $21H-5C$  & $-280689H+181531C$, $223H-53C$ &$4H-C$& -4& Both $|D|$ and $|-D|$ are empty, then use R-R\\
\hline
(28,23) &  $(2,4)$ &$(2,2,2)$   & $-8899156H+5413465C$, $4H-C$  & $-87646522H+53316447C$, $38H-9C$ &$4H-C$& -2& $|D|$ contains a smooth rational curve.\\
\hline

(4,8) &  $(3,2,2)$ &$(3,2,1,1)$   & $-H+2C$, $2H-C$  & $-2H+5C$, $5H-2C$ &$2H-C$& -2& $|D|$ contains a smooth rational curve.\\
\hline 

(4,10) &  $(2,2,2,2)$ &$(2,2,2,1,1)$   & $-38H+109C$, $2H-C$  & $-137H+393C$, $7H-3C$ &$2H-C$& -2& $|D|$ contains a smooth rational curve.\\
\hline

(6,11) &  $(2,2,2,2)$ &$(2,2,2,1,1)$   & $-258H+449C$, $2H-C$  & $-1652H+2875C$, $12H-5C$ &$2H-C$& -2& $|D|$ contains a smooth rational curve.\\
\hline

\end{tabular}

      \end{center}
    \end{adjustwidth}
\end{table}

As pointed out before, the complete intersection K3 surfaces $X$ do not have -2 divisors in Table 2 but have -2 divisors in Table 3.

\newpage
\section{Summary for the existence of smooth isolated curves in general CICY threefolds known so far}

\vspace{3mm}

Figure 1. Existence of smooth isolated curves in general $Y=(5)\subset\mathbb{P}^4$
\vspace{3mm}

\includegraphics{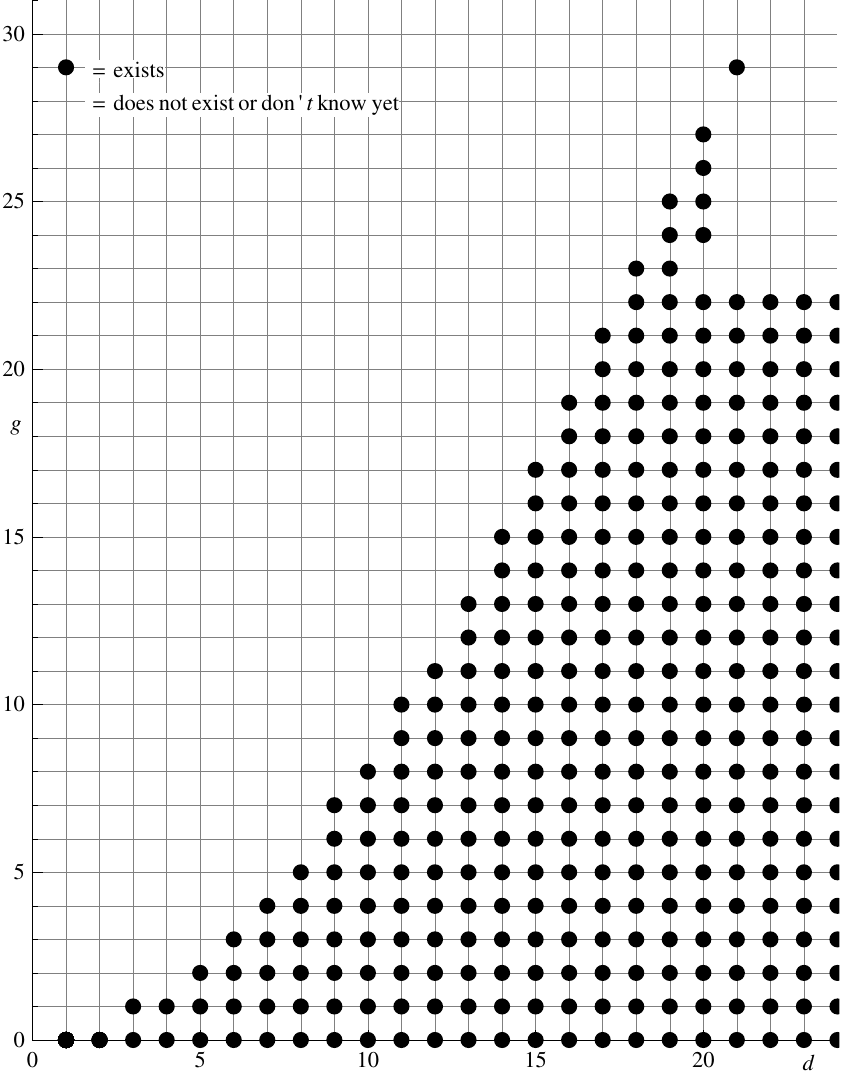}

Figure 2. Existence of smooth isolated curves in general $Y=(3,3)\subset\mathbb{P}^5$
\vspace{2mm}

\includegraphics{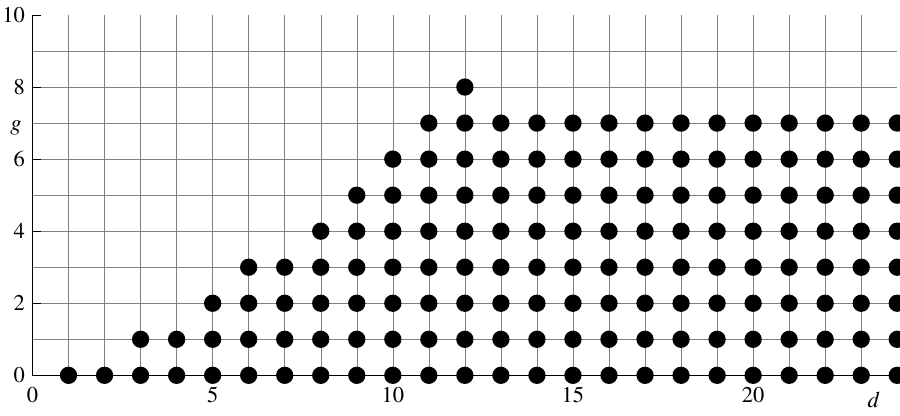}

Figure 3. Existence of smooth isolated curves in general $Y=(2,4)\subset\mathbb{P}^5$
\vspace{2mm}

\includegraphics{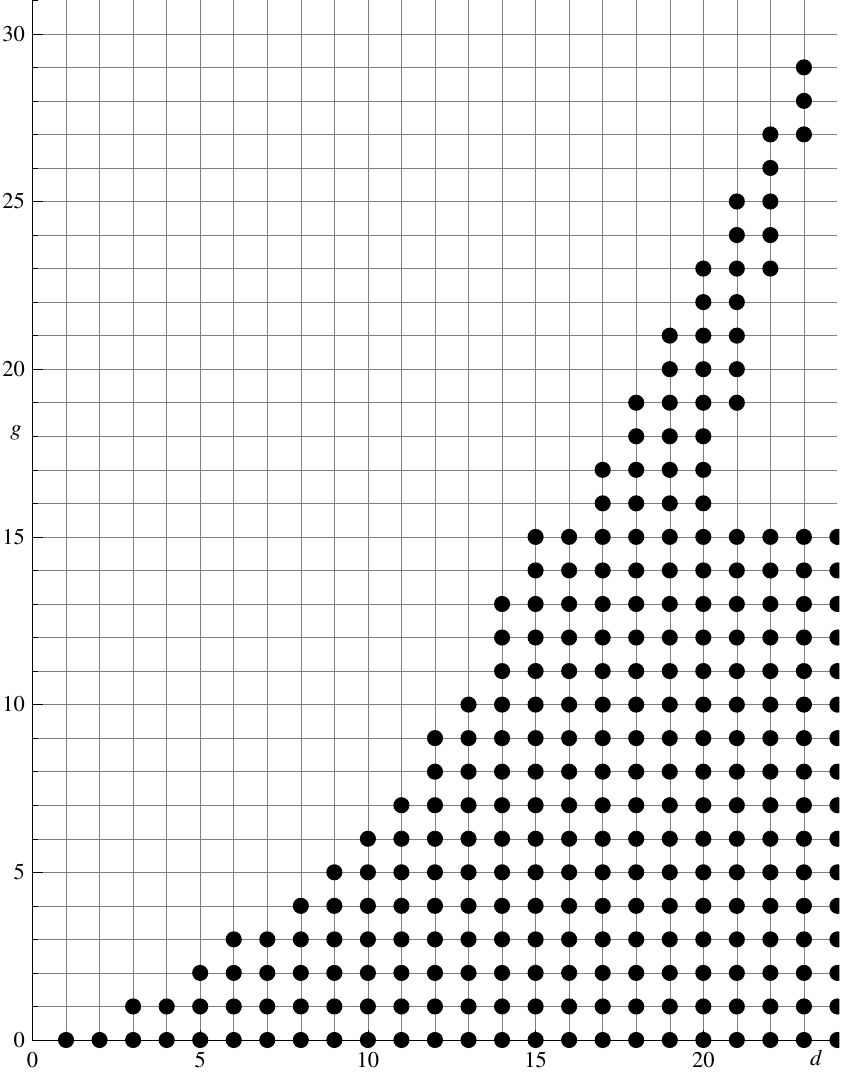}

Figure 4. Existence of smooth isolated curves in general $Y=(2,2,3)\subset\mathbb{P}^6$
\vspace{2mm}

\includegraphics{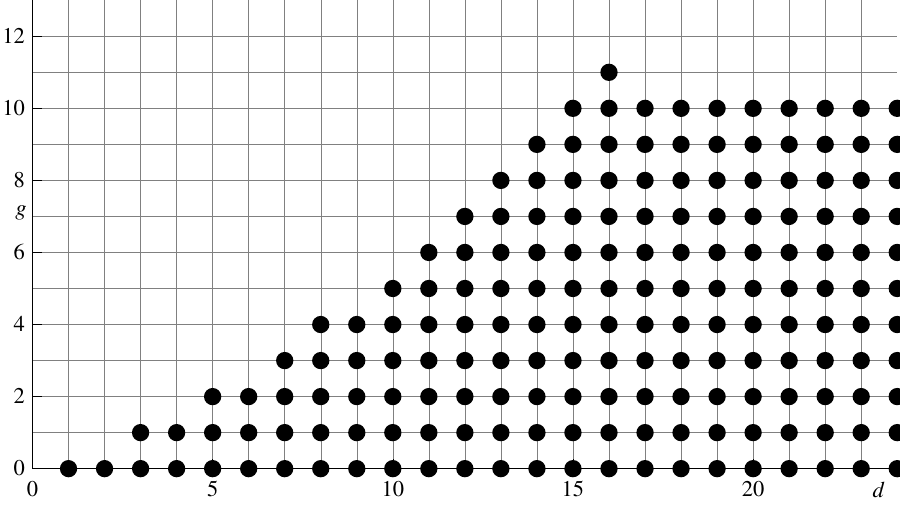}

\newpage
Figure 5. Existence of smooth isolated curves in general $Y=(2,2,2,2)\subset\mathbb{P}^7$
\vspace{2mm}

\includegraphics{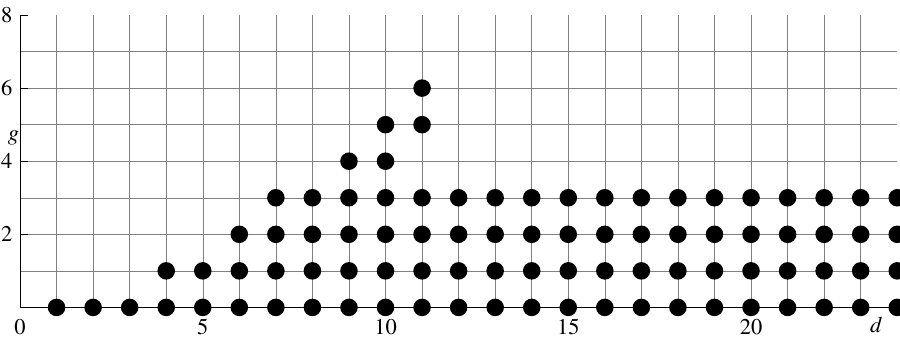}

\end{document}